\documentclass[a4paper,11pt,reqno]{amsart}
\pdfoutput=1

\usepackage[english]{babel}
\usepackage{xcolor}
\usepackage[pdfencoding=unicode, hidelinks]{hyperref}
\hypersetup{
	colorlinks=false,
	linkcolor={red},
	citecolor={blue},
}
\usepackage{crossreftools}
\usepackage{etoolbox}
\cslet{blx@noerroretextools}\empty 
\usepackage{mathtools}
\usepackage{mathrsfs}
\usepackage{amsthm}
\usepackage[nameinlink]{cleveref}
\usepackage{amsfonts}
\usepackage{amssymb}
\usepackage{amsmath}
\usepackage{autonum}
\usepackage[maxnames=50, sorting=ayt, doi=false, isbn=false, url=false, giveninits=true, date=year]{biblatex}
\AtEveryBibitem{\clearfield{number}}
\renewbibmacro{in:}{}
\usepackage{tikz}
\usepackage{orcidlink}
\usepackage{float}
\usepackage{stmaryrd}
\usepackage{graphics}
\usepackage{graphicx}
\usepackage{subfig}
\usepackage{cases}

\DeclareSortingTemplate{ayt}{
  \sort{
    \field{presort}
  }
  \sort[final]{
    \field{sortkey}
  }
  \sort{
    \field{sortname}
    \field{author}
  }
  \sort{
    \field{sortyear}
    \field{year}
    \literal{9999}
  }
  \sort{
    \field{sorttitle}
    \field{title}
  }
}

\usepackage[utf8]{inputenc}
\usepackage[shortlabels]{enumitem}
\usepackage[final]{microtype}
\usepackage{bm}
\usepackage{csquotes}
\usepackage[english]{babel}

\numberwithin{equation}{section}
\newtheorem{teo}{Theorem}[section]
\newtheorem{defi}[teo]{Definition}
\newtheorem{prop}[teo]{Proposition}
\newtheorem{lemma}[teo]{Lemma}

\theoremstyle{definition}
\begingroup
\newtheorem{remark}[teo]{Remark}

\endgroup

\theoremstyle{remark}

\renewcommand{\r}{\mathbb{R}}

\newcommand{\n}{\mathbb{N}}

\newcommand{\restr}[2]{\left.\kern-\nulldelimiterspace#1\vphantom{\big|}\right|_{#2}}

\DeclareMathOperator{\AC}{AC}
\DeclareMathOperator{\BV}{BV}

\DeclareMathOperator{\inter}{int}

\DeclareMathOperator*{\esssup}{ess\, sup}

\Crefname{lemma}{lemma}{lemmas}
\Crefname{lemma}{Lemma}{Lemmas}
\Crefname{prop}{proposition}{propositions}
\Crefname{prop}{Proposition}{Propositions}
\Crefname{defi}{definition}{definitions}
\Crefname{defi}{Definition}{Definitions}
\Crefname{cor}{corollary}{corollaries}
\Crefname{cor}{Corollary}{Corollaries}
\Crefname{remark}{remark}{remarks}
\Crefname{remark}{Remark}{Remarks}
\Crefname{teo}{theorem}{theorems}
\Crefname{teo}{Theorem}{Theorems}
\Crefname{section}{Section}{Sections}
\Crefname{app}{Appendix}{Appendices}
\Crefname{example}{Example}{Examples}

\title[A free boundary approach to quasistatic debonding]{A free boundary approach to the quasistatic evolution of debonding models}
\author[E. Maggiorelli, F. Riva, E. G. Tolotti]{Eleonora Maggiorelli\hspace{.5mm}\orcidlink{0009-0009-7986-2087}, Filippo Riva\hspace{.5mm}\orcidlink{0000-0002-7855-1262}, and Edoardo Giovanni Tolotti\hspace{.5mm}\orcidlink{0009-0004-3279-245X}}
\begin{document}

\begin{abstract}
		The mechanical process of progressively debonding an adhesive membrane from a substrate is described as a quasistatic variational evolution of sets and herein investigated. Existence of energetic solutions, based on global minimisers of a suitable functional together with an energy balance, is obtained within the natural class of open sets, improving and simplifying previous results known in literature. The proposed approach relies on an equivalent reformulation of the model in terms of the celebrated one-phase Bernoulli free boundary problem. This point of view allows performing the Minimizing Movements scheme in spaces of functions instead of the more complicated framework of sets. Nevertheless, in order to encompass irreversibility of the phenomenon, it remains crucial to keep track of the debonded region at each discrete time-step, thus actually resulting in a coupled algorithm.
	\end{abstract}
	
	\maketitle
	
	{\small
		\par
	}

	\pagenumbering{arabic}
	
	\medskip
	
	\tableofcontents
	
\renewcommand{\theequation}{I.\arabic{equation}}
\section*{Introduction}

Since the pioneeristic work of De~Giorgi~\cite{DeGiorgi}, the tool of Minimizing Movements has been widely employed in the analysis of variational evolutions, especially due to its application to problems with no linear structure.
In particular, it provides a robust technique which allows dealing with geometric flows or more generally with evolutions of sets.
Among the extensive pertaining literature, we mention the groundbreaking work of Almgren, Taylor, and Wang~\cite{ATW93} on the mean curvature flow, and a recent paper~\cite{DGChambolleMorini} which analyses its anisotropic version.
We also refer to~\cite{BucButtStef,CardLey} for set evolutions in shape optimization, and to~\cite{SciannaTilli} for an application to the Hele--Shaw flow.
The great flexibility of Minimizing Movements makes them suitable to tackle also quasistatic (sometimes called rate-independent) problems, namely when the system under consideration evolves through states of equilibrium.
We quote the manuscript~\cite{MielRoubbook}, where a precise and thorough presentation on rate-independent systems is provided, and we refer to~\cite{ChambolleNovaga, RossStefThom} for an investigation on quasistatic evolutions of sets driven by the perimeter functional, and to~\cite{BourdinFrancMar,DalMasoToader} for an application to fracture mechanics.

In this paper, we are interested in quasistatic evolutions of shapes modelling the mechanical phenomenon of debonding, often also called peeling.
Although the formulation of the problem is quite understood and rather old in time (see \cite{BurridgeKeller,Hellan}), a complete mathematical treatment of debonding processes is still missing, due to their underlying geometric complexity.
Up to our knowledge, only partial results are available in literature: we mention for instance \cite{BUCUR2008} where a relaxed version of the problem in terms of measures is investigated, or~\cite{RivQuas} for the simple one-dimensional setting.
The interested Reader may also see~\cite{LMRSMovingdomains} for the formulation of dynamic debonding models, namely when also inertial effects are taken into account, and~\cite{DMLN,LMS, RivContDep, RivNar} where dynamic problems are solved and analysed in dimension one or assuming a priori radial symmetry.

We now formally describe the quasistatic problem we intend to analyse, postponing the rigorous formulation to \Cref{sec:setting}. Consider a fixed open set $\Omega$ in $\r^d$ which represents the reference configuration of an adhesive elastic membrane stuck to a substrate; the physical dimension of the model is clearly $d=2$ (or $d=1$), but from the mathematical viewpoint the problem can be stated (and solved) in an arbitrary dimension $d\in \n$. Let $t\mapsto w(t)$ be a given time-dependent prescribed vertical displacement acting on a portion $\Gamma$ of the boundary of $\Omega$. As the time advances, the map $w(t)$ forces the film to debond from the substrate, thus creating a debonded region $A(t)\subseteq\Omega$. In the quasistatic framework, the rules governing this phenomenon can be stated by introducing two ingredients: a time-dependent driving energy 
\begin{equation}\label{eq:elastic}
    \mathcal E(t,A)=\min\left\{\, \frac 12\int_\Omega |\nabla v|^2\, dx : \, v\in H^1(\Omega),\,v=w(t)\text{ on }\Gamma,\text{ and }v=0\text{ on }\Omega\setminus A\, \right\},
\end{equation}
modelling the (linearized) elastic energy possessed by the membrane, which is not glued and hence free to move on $A$, at time $t$, and a dissipation distance 
\begin{equation}\label{eq:dissdist}
    \mathcal D(B,A)=\int_{B\setminus A}\kappa\,dx+\chi_{A\subseteq B}
\end{equation}
between two possible debonded states $A$ and $B$. Above, the function $\kappa$ represents the toughness of adhesion between the membrane and the substrate, while
\begin{equation}
    \chi_{A\subseteq B}:=\begin{cases}
        0,&\text{if }A\subseteq B,\\
        +\infty,&\text{ otherwise}.
    \end{cases}
\end{equation}
Hence, the dissipation distance \eqref{eq:dissdist} encodes the quantity of energy that the membrane needs to spend in order for the debonded region to grow from $A$ to $B$, while enforcing a crucial property of the debonding process, namely irreversibility.

Following the by now consolidated energetic formulation~\cite{MielkeTheil}, we aim at finding an evolution of sets $t\mapsto A(t)$ fulfilling at every time a global stability condition and an energy(-dissipation) balance:
\begin{equation}\label{eq:ensol}
 \begin{cases}
    \displaystyle\mathcal E(t,A(t))\le \mathcal E(t,B)+\mathcal{D}(B,A(t)),\qquad\text{for all }B\subseteq \Omega,\\
    \displaystyle\mathcal E(t,A(t))+\operatorname{Diss}_{\mathcal D}(A;[0,t])=\mathcal E(0,A(0))+\int_0^t P(\tau)\,d\tau.
\end{cases}
\end{equation}
The stability condition in the first line is saying that the solution wants to minimize the elastic energy of the membrane (i.e. the debonded region wants to grow), but dissipating as little energy as possible. At the same time, the energy-balance ensures that $t\mapsto A(t)$ is non-decreasing with respect to inclusion and that the sum between free energy and dissipated energy
\begin{equation}
    \operatorname{Diss}_{\mathcal D}(A;[0,t]):=\sup\limits_{\substack{{\text{finite partitions}}\\{\text{of }[0,T]}}}\sum_{k=1}^N\mathcal D(A(t_k),A(t_{k-1}))=
         \int_{A(t)\setminus A(0)}\kappa\,dx
\end{equation}
is conserved, up to an integral term representing the work of the prescribed displacement $w$. The integrand $P$ represents the power of $w$; in smooth situations there usually holds $P(t)=\partial_t\mathcal E(t,A(t))$, but in our setting a more involved expression is needed. We refer to Definition \ref{defi:SES}, in particular \eqref{eq:EB}, for a rigorous formula.

Analogous quasistatic debonding problems have been studied in~\cite{BUCUR2008} (see also~\cite{BucButt}), looking for solutions $t\mapsto A(t)$ living within the family of quasi-open sets, a natural class for shape minimization problems~\cite{ButtazzoDM} but whose definition intrinsically needs the theory of capacity. Their argument is based on a proper relaxation of \eqref{eq:ensol} to the space of capacitary measures, i.e. measures vanishing on sets of null capacity, in order to gain the necessary compactness needed for the use of Minimizing Movements. In this way the authors prove existence of an energetic solution of (irreversible) capacitary measures. However, although very general and suitable to deal with different dissipation distances than \eqref{eq:dissdist}, the theory developed in~\cite{BUCUR2008} does not provide complete results in terms of shapes: in~\cite[Section 5]{BUCUR2008} it is indeed shown how to build an evolution of quasi-open sets fulfilling the global stability condition, but it is also pointed out that no information on the energy balance are available, unless strong geometric assumptions are in force (see~\cite[Remarks 3--4]{BUCUR2008}).

In this paper, we propose an alternative approach which allows solving problem \eqref{eq:ensol} even within the simple family of open sets, thus making the arguments more accessible and less technical as possible.
It is based on the observation, often used in shape optimization~\cite{BrianconLamboley, Bucur} (see also the monograph~\cite{Henrot}, specifically Chapter 3), that the energy \eqref{eq:elastic} is actually defined by minimizing the Dirichlet energy among scalar functions (the possible displacements of the membrane). From this point of view, the debonded region can be formally identified with the positivity set of the minimiser (actually, this identification is not completely true since irreversibility may be violated, see the correct formula \eqref{eq:debsets} below). This leads to an equivalent reformulation of problem \eqref{eq:ensol} in terms of displacements: we now look for an evolution of functions $t\mapsto u(t)$ satisfying for all times
\begin{equation}\label{eq:reformulation}
 \begin{cases}
    \displaystyle\frac 12 \int_\Omega |\nabla u(t)|^2 dx\le \frac 12 \int_\Omega |\nabla v|^2 dx+\int_{\{v>0\}\setminus A_{u(t)}}\kappa\, dx,\quad \forall \,v\in H^1(\Omega),\,v=w(t) \text{ on }\Gamma,\\
    \displaystyle\frac 12 \int_\Omega |\nabla u(t)|^2 dx+\int_{A_{u(t)}\setminus A_{u(0)}}\kappa\,d x=   \frac 12 \int_\Omega |\nabla u(0)|^2 dx+\int_0^t P(\tau)\,d\tau,
\end{cases}
\end{equation}
where in this setting the debonded region takes the form
\begin{equation}\label{eq:debsets}
        A_{u(t)}=\bigcup\limits_{s\in [0,t]}\{u(s)>0\},
    \end{equation}
namely, it is the union of all the positivity sets of the solution $u$ at previous times. Notice that formula \eqref{eq:debsets} defines an open set provided the functions $u(t)$ are continuous, as it will be the case.

Besides the clear simplification yielded by working with functions instead of sets, the main advantage of the displacement reformulation \eqref{eq:reformulation} relies on its structure: indeed, it can be seen as the quasistatic version of a free boundary problem of Alt--Caffarelli type~\cite{ALT1981}, also known as one-phase Bernoulli problem. We quote the recent monograph~\cite{Velichkov} for an exhaustive introduction on the topic. This interpretation allows us to employ existing results well-known in the literature in order to make the machinery of Minimizing Movements fully work. Our approach thus provides existence of displacement energetic solutions to \eqref{eq:reformulation}, whence existence of energetic evolutions of open sets to the original debonding problem \eqref{eq:ensol} is obtained via the formula \eqref{eq:debsets}.

We finally mention that similar quasistatic free boundary problems have been recently studied in~\cite{CollinsFeld, FeldKimPoz2, FeldKimPoz} in the context of droplet evolution in capillarity models. In this framework, $u$ represents the graph of the free boundary surrounding the droplet while its positivity set describes the wetted area. Although related, such model presents crucial differences with respect to the debonding formulation \eqref{eq:reformulation}, which directly affect the techniques adopted in the Minimizing Movements scheme. Firstly, droplet evolution is reversible, whence the whole information is carried out by the positivity set of $u$  in place of \eqref{eq:debsets}. This feature is encoded in the mathematical formulation by considering a term of the form 
\begin{equation}
    \int_{A\setminus B}\widetilde\kappa \, dx
\end{equation}
in place of the constraint $\chi_{A\subseteq B}$ in the dissipation distance \eqref{eq:dissdist}. Reversibility also simplifies the Minimizing Movements algorithm since there is no need to \lq\lq separate variables\rq\rq, as we instead do in \eqref{eq:MMS1}--\eqref{eq:MMS2} (creating additional difficulties in the limit passage as the time-step vanishes) in order to keep track of the monotonicity of the debonded region. A second difference consists in the choice of the free energy. In our debonding model \eqref{eq:reformulation} it is simply given by the Dirichlet energy, while for describing droplet evolution it needs to be augmented with the measure of the positivity set itself. As it is shown in \cite{FeldKimPoz}, this structure yields uniform bounds on the perimeter of the positivity sets, implying good compactness properties that we lack in our setting (in the terminology of~\cite{Velichkov}, solutions to \eqref{eq:reformulation} are just outward minimisers of the Alt--Caffarelli functional, while droplets evolve throughout inward minimisers). We overcome this lack of (strong) compactness by exploiting irreversibility and relying on an appropriate version of Helly's Theorem, see Proposition~\ref{prop:rho} and Definition~\eqref{eq:definition_esssupp}.

\medskip
	
\noindent\textbf{Plan of the paper.} In Section \ref{sec:preliminaries} we fix the notations we adopt throughout the whole paper, and we list some properties possessed by minimisers of constrained Dirichlet energies and of the one-phase Bernoulli free boundary problem. Section \ref{sec:setting} is devoted to the rigorous formulation of the debonding problem \eqref{eq:ensol} and to its equivalent rewriting in terms of displacements \eqref{eq:reformulation}. We also state our main existence result in Theorem \ref{thm:safe} (see also Theorem \ref{thm:general}). For the sake of completeness, the simple one-dimensional framework is described. Section \ref{sec:properties} contains some useful properties of (displacement) energetic solutions to the debonding problem which we exploit for the proof of our main result, developed in Section \ref{sec:existence}.

\renewcommand{\theequation}{\thesection.\arabic{equation}}
\section{Notations and preliminaries}\label{sec:preliminaries} 

 Given a set $\Omega\subseteq \r^d$, $d\in \n$, we denote by $\mathcal M(\Omega)$ the family of Lebesgue-measurable subsets of $\Omega$. The Lebesgue measure of $A\in \mathcal M(\Omega)$ is denoted by $|A|$. For $A,B\in \mathcal M(\Omega)$, with a slight abuse of notation we write $A\subseteq B$ if the inclusion is true up to sets of null measure, i.e. $|A\setminus B|=0$. Analogously, we write $A=B$ if $|A\setminus B|=|B\setminus A|=0$. If a property holds almost everywhere in $\Omega$ with respect to the Lebesgue measure, we often adopt the customary shortcut \lq\lq a.e. in $\Omega$\rq\rq.
 
 For any family of scalar functions $X(\Omega)$, its subset of non-negative elements is indicated by $X(\Omega)^+$. We adopt standard notations for Lebesgue, Sobolev and Bochner spaces.
 
 If $\Omega$ is open with Lipschitz boundary, and $\Gamma \subseteq \partial \Omega$ has positive Hausdorff measure $\mathcal{H}^{d-1}(\Gamma) > 0$, given $\eta\in H^1(\Omega)$ and $A\in \mathcal M(\Omega)$ we introduce the expressions
 \begin{align}
 H^1_{\Gamma,\eta}(\Omega)&:=\{\varphi\in H^1(\Omega):\, \varphi=\eta \text{ on }\Gamma \},\\
     H^1_{\Gamma,\eta}(\Omega,A)&:=\{\varphi\in H^1_{\Gamma,\eta}(\Omega):\, \varphi=0 \text{ a.e. on }\Omega\setminus A\}.
 \end{align}
Observe that $H^1_{\Gamma,\eta}(\Omega,A)$ is a closed convex subset of $H^1(\Omega)$ and that $H^1_{\Gamma,\eta}(\Omega,A)=H^1_{\Gamma,\eta}(\Omega)$ whenever $|\Omega\setminus A|=0$.
 
 In view of the applications to debonding models, it will be useful to define the classes of admissible measurable and open sets with respect to the boundary value $\eta$ on $\Gamma$. This is done as follows:
 \begin{align}
     \mathcal M_{\Gamma,\eta}&:=\{A\in \mathcal M(\Omega):\,  H^1_{\Gamma,\eta}(\Omega,A)\text{ is nonempty}\},\\
     \mathcal O_{\Gamma,\eta}&:=\{A\in \mathcal M_{\Gamma,\eta}:\, A\text{ is open}\}.
 \end{align}

Finally, for any $A\in \mathcal M(\Omega)$ we denote by $A^+_{\Gamma,\eta}$ the union of all the connected components of $A$ whose boundary contains a subset of $\Gamma$ of positive Hausdorff measure on which $\eta>0$.

\subsection{minimisers of one-phase free boundary functionals}
We collect here some useful results regarding the main properties of minimisers of Dirichlet-type energies, including some well-known fact about the celebrated one-phase Bernoulli free boundary problem first analysed from a variational viewpoint by Alt and Caffarelli~\cite{ALT1981}. We also provide some additional notation which will be needed in the next section to rigorously introduce the quasistatic formulation of the debonding model under consideration.

This first elementary lemma deals with minimisers of the standard Dirichlet energy among functions vanishing outside a given set $A$. Here and henceforth, we tacitly require that $\Omega$ is open, connected, bounded, with Lipschitz boundary and that $\Gamma \subseteq \partial \Omega$ has positive Hausdorff measure.

\begin{lemma}\label{lemma:minimiser_dirichlet}
    Let $\eta\in H^1(\Omega)$ and $A \in \mathcal M_{\Gamma,\eta} $.  There exists a unique minimiser, denoted by $\mathfrak{h}_{A,\eta}$, of the Dirichlet functional
    \[
      D(v) = \frac 12 \int_\Omega |\nabla v|^2 \, dx \,
    \]
    over the set $H^1_{\Gamma,\eta}(\Omega,A)$. The minimiser $\mathfrak{h}_{A, \eta}$ is equivalently characterized by the Euler-Lagrange equation
    \begin{equation}\label{eq:euler_lagrange_dirichlet}
        \int_\Omega \nabla \mathfrak{h}_{A, \eta} \cdot \nabla \phi \, dx  = 0\quad \text{for all } \, \phi \in H^1_{\Gamma, 0}(\Omega, A) \,. 
    \end{equation}
    If in addition $A$ is open, i.e. $A\in\mathcal O_{\Gamma,\eta}$, then $\mathfrak{h}_{A,\eta}$ is harmonic in $A$; furthermore, if $\eta\in H^1(\Omega)^+$, then $0\leq \mathfrak{h}_{A,\eta} \leq \esssup_\Gamma\eta$ a.e. in $\Omega$ and there holds $\{\mathfrak{h}_{A,\eta}>0\}=A^+_{\Gamma,\eta}$.
\end{lemma}

\begin{proof}
    Since $H^1_{\Gamma,\eta}(\Omega, A)$ is non-empty by assumption, the existence of the minimum can be proved via the Direct Method of the Calculus of Variations.
    Uniqueness follows from the strict convexity of $D$ and the convexity of $H^1_{\Gamma,\eta}(\Omega, A)$.
    For any $\varepsilon > 0$ and $\phi \in H^1_{\Gamma,0}(\Omega, A)$, by using $\mathfrak{h}_{A,\eta}+\varepsilon \phi$ as a competitor, one easily deduces the validity of the Euler-Lagrange equation \eqref{eq:euler_lagrange_dirichlet}, which thus characterizes the minimiser $\mathfrak{h}_{A,\eta}$ by convexity. 
    
    If $A$ is open, it immediately follows by \eqref{eq:euler_lagrange_dirichlet} that $\mathfrak{h}_{A,\eta}$ is harmonic in $A$. If $\eta$ is non-negative, the Maximum Principle now yields $0\leq \mathfrak{h}_{A,\eta} \leq \esssup_\Gamma\eta$ a.e. in $\Omega$, and additionally the Strong Maximum Principle implies that $\{\mathfrak{h}_{A,\eta}>0\}=A^+_{\Gamma,\eta}$.
\end{proof}

Since the boundary of $A$ may be very wild, global regularity of $\mathfrak{h}_{A,\eta}$ is not available in general. However, better properties can be obtained for minimisers of the following functional, which can be seen as a regularization of the constraint $u=0$ outside $A$ since it penalises therein the measure of the positivity set of $u$. In particular, a free boundary is expected to appear.

Given $\kappa\in L^\infty(\Omega)^+$, we consider the Alt--Caffarelli functional $\mathcal{AC}\colon H^1(\Omega)\times\mathcal M(\Omega)\to [0,+\infty)$ defined as:
\begin{equation}\label{eq:E}
\mathcal{AC}(u,A):=D(u)+\int_{\{u>0\}\setminus A} \kappa  \,dx = \frac 12 \int_\Omega |\nabla u|^2 \, dx+\int_{\{u>0\}\setminus A} \kappa  \,dx \, .
\end{equation}

Setting $Q_A=\sqrt{2\kappa}(1-\chi_A)$, we observe that
\[
    2\mathcal{AC}(u,A)= \int_\Omega |\nabla u|^2 \, dx+\int_{\{u>0\}} Q_A^2  \,dx \,,
\]
hence the functional can be seen as a (one-phase) Bernoulli free boundary problem, in the notation of~\cite{ALT1981}.
Since $Q_A^2$ is bounded from above, we may exploit some of the results contained in~\cite{ALT1981} (see also the monograph~\cite{Velichkov}), deducing the following properties.

\begin{lemma}\label{lemma:minimiser_AltCaffarelli}
Let $\eta\in H^1(\Omega)^+$ and let $A\in \mathcal M(\Omega)$. Then the functional $v\mapsto \mathcal{AC}(v,A)$ admits a minimum in the class $H^1_{\Gamma,\eta}(\Omega)$. Moreover, any minimiser $u$ is locally Lipschitz continuous in $\Omega$ and for all $\Omega'$ well contained in $\Omega$ there exists a constant $C'>0$, independent of $A$, such that:
\begin{equation}
    \|u\|_{C^{0,1}(\Omega')}\leq C'(1+\esssup_\Gamma\eta).
\end{equation}
Furthermore, $u$ is harmonic in the open set $\{u>0\}$ and there holds $0\leq u\leq \esssup_\Gamma\eta$ in $\Omega$. Finally, one has $(\inter A)^+_{\Gamma,\eta}\subseteq\{u>0\}$, where $\inter A$ denotes the interior of $A$.
\end{lemma}

\begin{proof} 
    The only property not proved in~\cite{ALT1981} is the fact that $(\inter A)^+_{\Gamma,\eta}\subseteq\{u>0\}$. To show it, let us define the open set $B:= (\inter A)^+_{\Gamma,\eta} \cup \{u>0\}$. By definition, one has $u\in H^1_{\Gamma, \eta}(\Omega, B)$ and hence $B\in \mathcal O_{\Gamma,\eta}$. Since the function $\mathfrak{h}_{B,\eta}$ belongs to $H^1_{\Gamma,\eta}(\Omega)$, we know that $\mathcal{AC}( u, A)\leq \mathcal{AC}(\mathfrak{h}_{B,\eta}, A)$.
    Moreover, by \Cref{lemma:minimiser_dirichlet} and by the definition of $B$ we have
    \begin{equation}
        \{ \mathfrak{h}_{B,\eta}>0\} \setminus A \subseteq B \setminus A \subseteq \{ u>0\}\setminus A \,.
    \end{equation}
    Thus, it follows that
    \begin{equation}
        D(u)\le D(\mathfrak{h}_{B,\eta}).
   \end{equation}
    By the uniqueness of the minimiser proved in \Cref{lemma:minimiser_dirichlet}, we finally infer that $u$ coincides with $\mathfrak{h}_{B,\eta}$, and so $\{u > 0\}=\{\mathfrak{h}_{B,\eta} > 0\}=B^+_{\Gamma,\eta}\supseteq(\inter A)^+_{\Gamma,\eta}$ .
\end{proof}

\begin{remark}
    Note that $Q^2_A$ vanishes on $A$, hence $Q_A$ is not bounded below by a positive constant and so not all the results of~\cite{ALT1981} apply. For instance, the positivity set $\{u>0\}$ may not be of (locally) finite perimeter.
\end{remark}

Next lemma lists some simple, but useful for the rest of the paper, properties which are equivalent for a function $u$ to be a minimiser of $\mathcal{AC}(\cdot, A)$, knowing a priori that $u$ vanishes outside $A$.
Notice that condition~\eqref{item:equiv4} below just implies the outward minimality property in the sense of~\cite{Velichkov}, see for instance Chapter~3.

\begin{lemma} \label{lemma:equivalent}
    Let $\eta\in H^1(\Omega)^+$ and let $A\in \mathcal M_{\Gamma,\eta}$. For $u \in H^1_{\Gamma,\eta}(\Omega,A)$ the following are equivalent:
    \begin{enumerate}
        \item $\mathcal{AC}(u,A)\le \mathcal{AC}(v,A),\qquad$ for all $v\in H^1_{\Gamma,\eta}(\Omega)$; \label{item:equiv1}
        \item $\mathcal{AC}(u,A)\le \mathcal{AC}(v,A),\qquad$ for all $v\in H^1_{\Gamma,\eta}(\Omega)\cap C^0(\Omega)$;\label{item:equiv2}

        \item $\displaystyle \frac 12\int_\Omega |\nabla u|^2 \, dx  \leq \frac 12\int_\Omega |\nabla v|^2 \, dx + \int_{\{v > 0\}\setminus A}\kappa\, dx,\qquad$ for all $v\in H^1_{\Gamma,\eta}(\Omega)$;\label{item:equiv3}
        \item $\displaystyle \frac 12\int_\Omega |\nabla u|^2 \, dx  \leq \frac 12\int_\Omega |\nabla v|^2 \, dx + \int_{\{v > 0\}\setminus A}\kappa\, dx,\qquad$ for all $v\in H^1_{\Gamma,\eta}(\Omega)^+\cap C^0(\Omega)$ satisfying $A\subseteq \{v>0\}$.\label{item:equiv4}
    \end{enumerate}
    If one of the above holds, then $u=\mathfrak{h}_{A,\eta}$ and furthermore it possesses all the properties stated in Lemma \ref{lemma:minimiser_AltCaffarelli}.
\end{lemma}
\begin{proof} 
    \Cref{lemma:minimiser_AltCaffarelli} yields the equivalence between \eqref{item:equiv1} and \eqref{item:equiv2}. Since $u$ belongs to $H^1_{\Gamma,\eta}(\Omega,A)$ by assumption, the set $\{u>0\}$ is contained in $A$; this implies
    \begin{equation}
       \mathcal{AC}(u,A)=\frac 12 \int_\Omega |\nabla u|^2 \, dx, 
    \end{equation}
    and so \eqref{item:equiv1} is also equivalent to \eqref{item:equiv3}.

    Trivially \eqref{item:equiv3} implies \eqref{item:equiv4}, so we just need to prove the reverse implication; equivalently, we will actually show that \eqref{item:equiv4} implies \eqref{item:equiv2}.

    To this aim, let us fix $v\in  H^1_{\Gamma,\eta}(\Omega)\cap C^0(\Omega)$. By outer regularity of the Lebesgue measure, for any $\varepsilon >0$ there exists an open set $A_\varepsilon\subseteq \Omega$ containing $A$ and satisfying $|A_\varepsilon\setminus A|\le \varepsilon$. We now consider the 1-Lipschitz function $\varphi_\varepsilon(x):=\operatorname{dist}(x;\overline{\Omega}\setminus A_\varepsilon)$, which in particular satisfies $\{\varphi_\varepsilon>0\}=A_\varepsilon$, and we set $v_\varepsilon:=v^++\varepsilon\varphi_\varepsilon $. Observing that $v_\varepsilon\in H^1_{\Gamma,\eta}(\Omega)^+\cap C^0(\Omega)$ and that $A\subseteq A_\varepsilon\subseteq A_\varepsilon\cup \{v>0\}=\{v_\varepsilon>0\}$, by means of \eqref{item:equiv4} we obtain
    \begin{align}
        \mathcal{AC}(u,A)&=\frac 12 \int_\Omega |\nabla u|^2 dx\le \frac 12 \int_\Omega |\nabla v_\varepsilon|^2 dx+\int_{\{v_\varepsilon>0\}\setminus A}\kappa\, dx\\
        &=\frac 12 \int_\Omega |\nabla v^+|^2 dx+\frac{\varepsilon^2}{2}\int_\Omega |\nabla \varphi_\varepsilon|^2 dx+ \varepsilon\int_\Omega \nabla v^+\cdot\nabla \varphi_\varepsilon dx+ \int_{(\{v>0\}\cup A_\varepsilon)\setminus A}\kappa\, dx\\
        &\le \frac 12 \int_\Omega |\nabla v|^2 dx +\int_{\{v>0\}\setminus A}\kappa\, dx+\frac{\varepsilon^2}{2}|\Omega|+\varepsilon|\Omega|^{1/2}\|\nabla v\|_{L^2(\Omega)}+\varepsilon\|\kappa\|_{L^\infty(\Omega)}.
    \end{align}
    By sending $\varepsilon\to 0$ we deduce $\mathcal{AC}(u,A)\le \mathcal{AC}(v,A)$, namely \eqref{item:equiv2} holds true and equivalence between \eqref{item:equiv1}, \eqref{item:equiv2}, \eqref{item:equiv3} and \eqref{item:equiv4} is proved.

    Moreover, we notice that choosing $v=\mathfrak{h}_{A,\eta}$ in \eqref{item:equiv3}, since $\{\mathfrak{h}_{A,\eta}>0\}\subseteq A$ we obtain 
    \begin{equation}
        \frac 12\int_\Omega |\nabla u|^2 \, dx   \leq \frac 12 \int_\Omega |\nabla \mathfrak{h}_{A,\eta}|^2 \, dx \,,
    \end{equation}
    whence $u=\mathfrak{h}_{A,\eta}$  by uniqueness of the minimiser.
    Finally, due to \eqref{item:equiv1}, \Cref{lemma:minimiser_AltCaffarelli} applies to the function $u$.
\end{proof}

\section{Setting and main results}\label{sec:setting}

In this section we rigorously formulate the debonding problem under study as a quasistatic evolution of open sets. We also introduce the more manageable free boundary reformulation in terms of displacements, and we prove their equivalence. We then state our main results, regarding existence of this latter notion of solution. Finally, the simpler one-dimensional case, which provides explicit examples, is presented.

Let $\Omega \subseteq \r^d$ be an open, connected, bounded, and Lipschitz set representing the reference configuration of a linearly elastic adhesive film, and let $\Gamma \subseteq \partial \Omega$ with $\mathcal{H}^{d-1}(\Gamma) > 0$ be the portion of the boundary where a time-dependent vertical displacement $w(t)$, which drives the evolution of the system, is prescribed. As customary, we require that $w$ is the trace over $\Gamma$ of a function, still denoted by $w$,  of class $ AC([0,T];H^1(\Omega))$, where $T>0$ is an arbitrary horizon time, and we assume that it satisfies
\begin{equation}\label{eq:boundednessw}
    0\le w(t)\le M\quad\text{a.e. in }\Omega,\text{ for all }t\in [0,T].
\end{equation}

For $t\in [0,T]$, the free energy of a set $A\in \mathcal M_{\Gamma,w(t)}$, which will represent the debonded part of the film at time $t$, is the minimal Dirichlet energy of functions supported in $A$ attaining the correct boundary value $w(t)$ on $\Gamma$, namely
\begin{equation}
    \mathcal E(t,A):=\min\limits_{v \in H^1_{\Gamma,w(t)}(\Omega,A)}\frac 12 \int_\Omega|\nabla v|^2 \,dx.
\end{equation}
We recall that by \Cref{lemma:minimiser_dirichlet} one has 
\begin{equation}\label{eq:energy}
    \mathcal E(t,A)=\frac 12 \int_\Omega|\nabla \mathfrak{h}_{A,w(t)}|^2 \,dx.
\end{equation}

We also consider an initial debonded region $A_0\in \mathcal O_{\Gamma,w(0)}$, and a function $\kappa\in L^\infty(\Omega\setminus A_0)^+$ representing the toughness of adhesion of the film to the substrate. The energy dissipated during the debonding process growing from a debonded configuration $A$ to $B$ is thus modelled by the integral term
\begin{equation}
    \int_{B\setminus A}\kappa\, dx.
\end{equation}

Since the portion of film initially debonded is already encoded in the model through $A_0$, without loss of generality we may assume that the toughness $\kappa$ is positive on $\Omega\setminus A_0$. For the same reason, when needed, we will consider $\kappa$ as a function defined on the whole of $\Omega$, simply setting $\kappa:= 0$ on $A_0$.

In view of the notion of energetic solution we are going to introduce (see~\cite{MielRoubbook} for a general dissertation), it is convenient to require the initial set $A_0$ to be globally stable with respect to the energy $\mathcal E$ and the previously introduced dissipation, meaning that
\begin{equation}\label{eq:A0stable}
    \mathcal E(0,A_0)\le \mathcal E(0,B)+\int_{B\setminus A_0}\kappa\, dx,\qquad\text{for any }B\in \mathcal M_{\Gamma,w(0)} \text{ such that } A_0\subseteq B.
\end{equation}
\begin{remark}
    If $w(0)=0$ on $\Gamma$, we may even consider $A_0$ to be the empty set, namely at the initial time the film is completely glued to the substrate. Indeed, the identically $0$ function belongs to $H^1_{\Gamma,0}(\Omega,\emptyset)$.
    
    Moreover, we observe that in the case $w(0)=0$ on $\Gamma$, any set $A_0\in \mathcal M(\Omega)$ is stable in the sense of \eqref{eq:A0stable}, since $\mathcal E(0,A_0)=0$.
\end{remark}

The following notion of solution for the debonding model is based on three crucial considerations: the evolution of the sets is irreversible, indeed the debonded region can only increase; the process happens through minima of the total energy of the system, namely the state is globally stable at every time; along the evolution a balance between the free energy, the energy dissipated by debonding, and the energy inserted in the system via the work of the prescribed displacement $w$ is preserved.

\begin{defi}\label{defi:SES}
    Under the previous assumptions, we say that a set-valued map $[0,T]\ni t\mapsto A(t)$ is a \emph{shape energetic solution} (SES) of the debonding model if:
    \begin{itemize}[leftmargin = !, labelwidth=1.2cm, align = left]
        \item [{\crtcrossreflabel{\textup{(CO)}$_\text{S}$}[def:COS]}] $A(t)\in \mathcal O_{\Gamma,w(t)}$ for all $t\in [0,T]$;
        \item[{\crtcrossreflabel{\textup{(ID)}$_\text{S}$}[def:IDS]}] $A(0)=A_0$;
        \item[{\crtcrossreflabel{\textup{(IR)}$_\text{S}$}[def:IRS]}] $A(s)\subseteq A(t)$ for all $0\le s\le t\le T$;
        \item[{\crtcrossreflabel{\textup{(GS)}$_\text{S}$}[def:GSS]}] for all $t\in [0,T]$ there holds
        \begin{equation}
            \mathcal E(t,A(t))\le \mathcal E(t,B)+\int_{B\setminus A(t)}\kappa\, dx,\quad \text{for all $B\in \mathcal M_{\Gamma, w(t)}$ such that $A(t)\subseteq B$};
        \end{equation}
    \item[{\crtcrossreflabel{\textup{(EB)}$_\text{S}$}[def:EBS]}] $A(t)\in \mathcal O_{\Gamma,\dot w(t)}$ for a.e. $t\in [0,T]$, the map $\displaystyle t\mapsto \int_\Omega \nabla \mathfrak{h}_{A(t),\dot w(t)}\cdot \nabla \mathfrak{h}_{A(t),w(t)} \, dx$ belongs to $L^1(0,T)$ and for all $t\in [0,T]$ there holds
        \begin{equation}\label{eq:EB}
            \mathcal E(t,A(t))+\int_{A(t)\setminus A_0}\kappa\,d x=\mathcal E(0,A_0)+\int_0^t  \int_\Omega \nabla \mathfrak{h}_{A(\tau),\dot w(\tau)}\cdot \nabla \mathfrak{h}_{A(\tau),w(\tau)} \, dx\, d\tau.
        \end{equation}
    \end{itemize}
\end{defi}

\begin{remark}
    The last integral term in the energy balance \eqref{eq:EB} represents the work done by the external displacement $w$. Indeed, in the classical framework, it is given by the expression
    \begin{equation}
        \int_0^t  \int_\Gamma \dot w(\tau)\partial_n \mathfrak{h}_{A(\tau),w(\tau)} \, d\mathcal H^{d-1}\, d\tau,
    \end{equation}
    which by integration by parts formally coincides to the integral in \eqref{eq:EB} (exploiting the fact that the normal derivative $\partial_n \mathfrak{h}_{A(\tau),w(\tau)}$ necessarily vanishes in $\partial\Omega\setminus \Gamma$).

    We also stress that the more common expression 
    \begin{equation}\label{eq:workwrong}
        \int_0^t  \int_\Omega \nabla \dot w(\tau)\cdot \nabla \mathfrak{h}_{A(\tau),w(\tau)} \, dx\, d\tau \,,
    \end{equation}
    widely used for instance in fracture mechanics~\cite{DalMasoToader}, is clearly not always correct in our moving domain context. Indeed, consider $w(t)=t$ and $A_0=\emptyset$, and let $t\mapsto A(t)$ be an evolution satisfying the energy balance \eqref{eq:EB} with \eqref{eq:workwrong} as external work. Then it must hold
    \begin{equation}
        \mathcal E(t,A(t))+\int_{A(t)}\kappa\,d x=0\quad\text{for all }t\in [0,T],
    \end{equation}
    namely both terms above vanish. This yields that $A(t)=\emptyset$ and so that $\mathfrak{h}_{A(t),t}$ is identically $0$, which of course contradicts the boundary conditions $\mathfrak{h}_{A(t),t}=t$ on $\Gamma$. On the other hand, if one knew that $\dot w(t)=0$ outside $A(t)$ (see for instance \eqref{eq:safeassumption}), then the Euler-Lagrange equation \eqref{eq:euler_lagrange_dirichlet} would grant the equivalence between \eqref{eq:workwrong} and the last term of \eqref{eq:EB}.    
\end{remark}

Motivated by the expression \eqref{eq:energy}, we now reformulate the model in terms of displacements rather than sets, obtaining an evolutive version of the one-phase Bernoulli free boundary problem. We will focus on this formulation in the rest of the paper.

\begin{defi}\label{defi:DES}
    Under the previous assumptions, we say that a map $[0,T]\ni t\mapsto u(t)$ is a \emph{displacement energetic solution} (DES) of the debonding model if, setting
    \begin{equation}\label{eq:Au}
        A_{u(t)}:=A_0\cup\bigcup\limits_{s\in [0,t]}\{u(s)>0\},
    \end{equation}
    the following conditions hold:
    \begin{enumerate}[leftmargin = !, labelwidth=1cm, align = left]
        \item [{\crtcrossreflabel{\textup{(CO)}}[def:CO]}] $u(t)\in H^1_{\Gamma,w(t)}(\Omega)\cap C^0(\Omega)$ for all $t\in [0,T]$;
        \item[{\crtcrossreflabel{\textup{(ID)}}[def:ID]}] $u(0)=\mathfrak{h}_{A_0,w(0)}$;
        \item[{\crtcrossreflabel{\textup{(GS)}}[def:GS]}] for all $t\in [0,T]$ there holds
        \begin{equation}
            \frac 12 \int_\Omega |\nabla u(t)|^2 dx\le \frac 12 \int_\Omega |\nabla v|^2 dx+\int_{\{v>0\}\setminus A_{u(t)}}\kappa\, dx,\quad \text{for all }v\in H^1_{\Gamma,w(t)}(\Omega);
        \end{equation}
    \item[{\crtcrossreflabel{\textup{(EB)}}[def:EB]}] $A_{u(t)}\in \mathcal O_{\Gamma,\dot w(t)}$ for a.e. $t\in [0,T]$, the map $\displaystyle t\mapsto \int_\Omega \nabla \mathfrak{h}_{A_{u(t)},\dot w(t)}\cdot \nabla u(t) \, dx$ belongs to $L^1(0,T)$ and for all $t\in [0,T]$ there holds
        \begin{equation}\label{eq:EB2}
            \frac 12 \int_\Omega |\nabla u(t)|^2 dx+\int_{A_{u(t)}\setminus A_0}\kappa\,d x=   \frac 12 \int_\Omega |\nabla u(0)|^2 dx+\int_0^t \int_\Omega \nabla \mathfrak{h}_{A_{u(\tau)},\dot w(\tau)}\cdot \nabla u(\tau) \, dx\, d\tau.
        \end{equation}
    \end{enumerate}
\end{defi}

We now show that the two previous definitions of solution are somehow equivalent (see also \Cref{lemma:safe}). Anyway, as stated in the next proposition, notice that the notion of displacement energetic solution, which we will prove existence of, is always stronger than the notion of shape energetic solution.

\begin{prop}\label{prop: almostEQ}
    \Cref{defi:SES,defi:DES} are almost equivalent, in the following sense: 
    \begin{itemize}
        \item[(a)] if $t\mapsto u(t)$ is a DES, then $t\mapsto A_{u(t)}$ is a SES;
        \item[(b)] if $t\mapsto A(t)$ is a SES, then $t\mapsto \mathfrak{h}_{A(t),w(t)}$ fulfils \ref{def:CO}, \ref{def:ID} and \ref{def:GS}. If in addition the inclusion
        \begin{equation}\label{eq:hypinclusion}
            A(t)\setminus A_0\subseteq \bigcup_{s\in[0,t]}\{\mathfrak{h}_{A(s),w(s)}>0\}\setminus A_0,
        \end{equation}
        is in force, then also \ref{def:EB} holds true, hence $t\mapsto \mathfrak{h}_{A(t),w(t)}$ is a DES.
    \end{itemize}     
\end{prop}

\begin{proof}
    Assume that $t\mapsto u(t)$ is a DES and set $A(t):=A_{u(t)}$. The irreversibility condition \ref{def:IRS} is directly satisfied by the explicit form \eqref{eq:Au} of $A_{u(t)}$. Moreover, since by \Cref{lemma:minimiser_dirichlet} we know that $\{\mathfrak{h}_{A_0,w(0)}>0\}\subseteq A_0$, we easily deduce \ref{def:IDS}.

    By combining \ref{def:GS} and \Cref{lemma:minimiser_AltCaffarelli} we observe that $u(t)\in H^1_{\Gamma,w(t)}(\Omega,A(t))$, and so also \ref{def:COS} is satisfied. By combining \ref{def:GS} and \Cref{lemma:equivalent} instead, we now deduce that $u(t)=\mathfrak{h}_{A(t),w(t)}$. Thus, by \eqref{eq:energy}, the energy balance \ref{def:EBS} is just a rewriting of \ref{def:EB}.

    We just need to show the validity of \ref{def:GSS}. Let us fix $B\in \mathcal M_{\Gamma,w(t)}$ such that $A(t)\subseteq B$ and set $v:=\mathfrak{h}_{B,w(t)}$, so that in particular $\{v>0\}\subseteq B$. By \ref{def:GS} we then infer
    \begin{equation}
        \mathcal E(t,A(t))=\frac 12 \int_\Omega |\nabla u(t)|^2 \, dx\le \frac 12 \int_\Omega |\nabla v|^2 \, dx+\int_{\{v>0\}\setminus A(t)}\kappa\, dx\le \mathcal E(t,B)+\int_{B\setminus A(t)}\kappa\, dx,
    \end{equation}
    and so $(a)$ is proved.

    Let the map $t\mapsto A(t)$ be a SES and set $u(t):= \mathfrak{h}_{A(t),w(t)}$, so that \ref{def:ID} is automatically satisfied. Also observe that by \ref{def:IRS} one has
     \begin{equation}\label{eq:contained}
         A_{u(t)}=A_0\cup \bigcup_{s\in [0,t]}\{\mathfrak{h}_{A(s),w(s)}>0\}\subseteq A_0\cup \bigcup_{s\in [0,t]} A(s)=A(t),\qquad\text{for all }t\in [0,T]. 
     \end{equation}

    We now prove \ref{def:GS}, which by \Cref{lemma:equivalent} directly yields also \ref{def:CO}. Let $v\in H^1_{\Gamma,w(t)}(\Omega)$ and set $B:=A(t)\cup \{v>0\}$, which belongs to $\mathcal M_{\Gamma,w(t)}$ since by construction $v^+\in H^1_{\Gamma,w(t)}(\Omega,B)$. Noting that $A(t)\subseteq B$, by using \ref{def:GSS} and \eqref{eq:contained} we deduce
    \begin{align}
       \frac 12 \int_\Omega |\nabla u(t)|^2 \, dx&=  \mathcal E(t,A(t))\le \mathcal E(t,B)+\int_{B\setminus A(t)}\kappa\, dx\le \frac 12\int_\Omega |\nabla v^+|^2 \, dx+\int_{\{v>0\}\setminus A(t)}\kappa\, dx\\&
       \le \frac 12\int_\Omega |\nabla v|^2 \, dx+\int_{\{v>0\}\setminus A_{u(t)}}\kappa\, dx,
    \end{align}
    whence \ref{def:GS} holds true.

    Assume now \eqref{eq:hypinclusion}, so that we actually infer equality in \eqref{eq:contained}. Thus, the energy balance \ref{def:EB} is just a rewriting of \ref{def:EBS}, and we conclude.
\end{proof}

Although we believe condition \eqref{eq:hypinclusion} should be always true (whence DES and SES are actually equivalent) due to \ref{def:GSS} and \ref{def:EBS}, we are not able to prove it in the general situation. A sufficient condition which provides its validity (see \Cref{lemma:safe} below) is given by the following assumption, which will play an important role also for the existence of DES:  
\begin{equation}\label{eq:safeassumption}
    \text{there exists $\overline w\in AC([0,T];H^1(\Omega))$ satisfying $\overline w(t)\in H^1_{\Gamma,w(t)}(\Omega,A_0)^+$ for all $t\in[0,T]$.}
\end{equation}
Observe that if the initial debonded region $A_0$ contains a neighbourhood of $\Gamma$, then \eqref{eq:safeassumption} is fulfilled by $\overline w(t,x):=\Phi(x)w(t,x)$, where $\Phi\in C^1(\overline \Omega)$ is a suitable cut-off function satisfying $0\le \Phi\le 1$ in $\overline{\Omega}$, $\Phi=1$ in $\Gamma$, and $\Phi=0$ outside the neighbourhood of $\Gamma$.

For technical reasons, the proof of the following lemma is postponed to \Cref{sec:properties}, just after \Cref{prop:upperineq}.

\begin{lemma}\label{lemma:safe}
    Assuming \eqref{eq:safeassumption}, any SES $t\mapsto A(t)$ satisfies the inclusion \eqref{eq:hypinclusion}. In particular, the two notions of SES and DES are equivalent.
\end{lemma}

We are now in a position to state the main results of the paper, regarding existence of displacement energetic solutions of the debonding model, and so in particular existence of shape energetic solutions. We stress that uniqueness and time-regularity are not expected to hold in general, see the one-dimensional examples of \Cref{sec:1d}. The first theorem needs assumption~\eqref{eq:safeassumption}.

\begin{teo}\label{thm:safe}
    Assume that $w\in AC([0,T];H^1(\Omega))$ satisfies \eqref{eq:boundednessw}, and that the initial debonded region $A_0\in \mathcal O_{\Gamma,w(0)}$ fulfils \eqref{eq:A0stable}. If \eqref{eq:safeassumption} is in force, there exists a DES of the debonding model in the sense of \Cref{defi:DES}. In this case, the work of the prescribed displacement in the energy balance \eqref{eq:EB2} takes the simpler form
        \begin{equation}
        \int_0^t  \int_\Omega \nabla \dot{\overline w}(\tau)\cdot \nabla u(\tau) \, dx\, d\tau.
    \end{equation}
\end{teo}
\begin{remark}\label{rmk:nostable}
    If the stability at the initial time \eqref{eq:A0stable} is dropped, the very same proof performed in \Cref{sec:existence} still yields the existence of a map $t\mapsto u(t)$ fulfilling \ref{def:ID}, \ref{def:CO} and \ref{def:GS} just for $t\in (0,T]$, and the following energy inequality
    \begin{equation}
        \frac 12 \int_\Omega |\nabla u(t)|^2 dx+\int_{A_{u(t)}\setminus A_0}\kappa\,d x\le   \frac 12 \int_\Omega |\nabla u(0)|^2 dx+\int_0^t  \int_\Omega \nabla \dot{\overline w}(\tau)\cdot \nabla u(\tau) \, dx\, d\tau
    \end{equation}
    for all $t\in [0,T]$.
\end{remark}

The second result gets rid of assumption \eqref{eq:safeassumption}, thus for instance it may include the situation $A_0=\emptyset$, but in general it just ensures existence of an evolution of globally stable states, without any information on the energy balance \eqref{eq:EB}.
\begin{teo}\label{thm:general}
    Assume that $w\in AC([0,T];H^1(\Omega))$ satisfies \eqref{eq:boundednessw}, and that the initial debonded region $A_0\in \mathcal O_{\Gamma,w(0)}$ fulfils \eqref{eq:A0stable}. Then there exists a function $t\mapsto u(t)$ satisfying \ref{def:CO}, \ref{def:ID} and \ref{def:GS}.   
\end{teo}
The proof will be performed by approximating the initial debonded region fattening the set $A_0$. For $\varepsilon>0$ we consider the tubular neighbourhood $\Gamma^\varepsilon:=\{x\in \Omega :\, \operatorname{dist}(x;\Gamma)<\varepsilon\}$, and we set $A_0^\varepsilon:= A_0\cup \Gamma^\varepsilon$. Since $\Gamma^\varepsilon$ is open, we may construct a function $w^\varepsilon$ satisfying \eqref{eq:safeassumption} by multiplying $w$ with a suitable cut-off function.
Relying on \Cref{thm:safe}, we thus construct solutions $u^\varepsilon$ with boundary values $w^\varepsilon$ and initial debonded region $A^\varepsilon_0$, and we then show that as $\varepsilon\to 0$ we recover an evolution fulfilling the statement of \Cref{thm:general}. In addition, whenever a uniform global Lipschitz bound of the sequence $u^\varepsilon$ holds true, and under further regularity of the boundary datum, in the limit we actually retrieve a DES, as stated in the following proposition.
\begin{prop}
    \label{thm:conditional}Assume that $w\in AC([0,T];W^{2,p}(\Omega))$ with $p>d$ satisfies \eqref{eq:boundednessw}, and that the initial debonded region $A_0\in \mathcal O_{\Gamma,w(0)}$ fulfils \eqref{eq:A0stable}. Moreover, suppose that $\min\limits_\Gamma w(t)>0$ for a.e. $t\in [0,T]$. If the following uniform bound holds true
    \begin{equation}\label{eq:conditionalbound}
        \esssup\limits_{t\in [0,T]}\|u^\varepsilon(t)\|_{C^{0,1}(\Omega)}\le C,
    \end{equation}
    then there exists a DES of the debonding model in the sense of \Cref{defi:DES}. 
\end{prop} 
\begin{remark}
    Although it seems natural that the global Lipschitz estimate \eqref{eq:conditionalbound} should be a byproduct of the stability condition (if the boundary $\partial\Omega$ and the boundary datum $w$ are regular enough), we are not aware of results of this type for outer minimisers of the Alt--Caffarelli functional. Up to our knowledge, the closest available global regularity result is proved in \cite{fernandez2024continuity} (see also \cite[Appendix B]{Ede}), but it just provides global H\"older bounds. Unfortunately, as it is clear in the proof of \Cref{p:cond1}, these are not enough to pass to the limit the work of the prescribed displacement. 
\end{remark}

\subsection{One-dimensional framework}\label{sec:1d}

We write here the formulation of the debonding problem in one dimension, providing a simple framework where explicit examples are available. 

    Let $\Omega$ be the interval $(0,L)$ and set $\Gamma=\{0\}$. The case $\Gamma=\{0,L\}$ may be treated similarly. Consider the prescribed displacement $w\in AC([0,T])^+$ acting on $x=0$ and let $A_0\in \mathcal O_{\Gamma,w(0)}$. The toughness of adhesion is a function $\kappa\in L^{\infty}(0,L)^+$, positive outside $A_0$ and vanishing in $A_0$. We also define $\ell_0:=\sup\{\lambda\in [0,L]: (0,\lambda)\subseteq A_0\}$.

    \begin{prop}
        The map $t\mapsto u(t)$ is a DES of the debonding model if and only if it has the form
        \begin{equation}\label{eq:formu}
            u(t,x)=\begin{cases}
                0&\text{if }\ell(t)=0,\\
                w(t)\left(1-\frac{x}{\ell(t)}\right)^+ &\text{if }\ell(t)\in(0,L),\\
                w(t) &\text{if }\ell(t)=L,
            \end{cases}
        \end{equation}
        for some non-decreasing function $\ell\colon [0,T]\to [\ell_0,L]$, representing the debonding front, which satisfies:
        \begin{enumerate}[leftmargin = !, labelwidth=1cm, align = left] 
            \item [{\crtcrossreflabel{\textup{(CO)}$_\ell$}[def:COL]}] if $\ell(t)=0$, then $w(t)=0$;
            \item[{\crtcrossreflabel{\textup{(ID)}$_\ell$}[def:IDL]}] $\ell(0)=\ell_0$;
            \item[{\crtcrossreflabel{\textup{(GS)}$_\ell$}[def:GSL]}] if $\ell(t)\in (0,L)$, then
        \begin{align}
            &\frac 12 \frac{w(t)^2}{\ell(t)}\le \frac 12 \frac{w(t)^2}{\rho}+\int_{\ell(t)}^\rho \kappa\, dx\qquad\text{for all }\rho\in (\ell(t),L),\\
            &\frac 12 \frac{w(t)^2}{\ell(t)}\le \int_{\ell(t)}^L \kappa\, dx;
        \end{align}
    \item[{\crtcrossreflabel{\textup{(EB)}$_\ell$}[def:EBL]}] the function
            \[
                P(t):=\begin{cases}
                    0&\text{ if } \ell(t)=0 \text{ or }\ell(t)=L,\\
                    \dot w(t) \frac{w(t)}{\ell(t)}  &\text{ if }\ell(t)\in(0,L),
                \end{cases}
            \]
        belongs to $L^1(0,T)$ and for all $t\in [0,T]$ there holds
        \begin{equation}
            F(t,\ell(t))+\int_{\ell_0}^{\ell(t)}\kappa\, dx= F(0,\ell_0) +\int_0^t P(\tau)\, d\tau,
        \end{equation}
        where the energy $F\colon [0,T]\times [0,L]\to [0,+\infty)$ is defined as
        \begin{equation}
            F(t,\ell):=\begin{cases}
                0&\text{if } \ell=0 \text{ or }\ell=L,\\
                \frac 12 \frac{w(t)^2}{\ell} &\text{if } \ell\in (0,L).
            \end{cases}
        \end{equation}
    \end{enumerate}
    Moreover, one has
    \begin{equation}\label{eq:Au1d}
            A_{u(t)}=A_0\cup (0,\ell(t)), \qquad\text{for all }t\in [0,T].
        \end{equation}
    \end{prop}

    \begin{proof}
        Let first $t\mapsto u(t)$ be a DES, and for all $t\in [0,T]$ set $\ell(t):=\sup\{\lambda\in [0,L]: (0,\lambda)\subseteq A_{u(t)}\}$. Clearly $\ell$ is non-decreasing and \ref{def:IDL} holds true.

        We now observe that if $\ell(t)<L$, then necessarily one has $u(t,\ell(t))=0$. Otherwise, since $H^1(0,L)$ embeds in $C^0([0,L])$, by continuity $u(t)$ would be positive in a right neighbourhood of $\ell(t)$, thus contradicting the definition of $\ell(t)$. In particular, this argument shows that \ref{def:COL} holds true.

        Let us now prove the validity of \eqref{eq:formu}. If $\ell(t)=0$, by \ref{def:COL} we know that $w(t)=0$ as well, and so we can take $v=0$ as a competitor in \ref{def:GS}: this implies that $u(t)$ is constant, and so equal to $0$, in $[0,L]$. If $\ell(t)=L$, we can argue similarly: in this case $A_{u(t)}$ coincides with $(0,L)$, and so picking $v=w(t)$ in \ref{def:GS} yields that $u(t,x)=w(t)$. Finally, if $\ell(t)\in (0,L)$, we choose as a competitor $v=w(t)\left(1-\frac{x}{\ell(t)}\right)^+$: noting that $\{v>0\}\subseteq (0,\ell(t))\subseteq A_{u(t)}$ we deduce
        \begin{align}
            \frac 12\int_0^{\ell(t)}|\partial_x u(t)|^2\, dx&\le\frac 12\int_0^{L}|\partial_x u(t)|^2\, dx\le \frac 12\int_0^{L}|\partial_x v|^2\, dx+\int_{\{v>0\}\setminus A_{u(t)}}\kappa \, dx=\frac 12\int_0^{\ell(t)}|\partial_x v|^2\, dx\\
            &\le \frac 12\int_0^{\ell(t)}|\partial_x u(t)|^2\, dx,
        \end{align}
        where in the last inequality we exploited the fact that $v$ minimizes the Dirichlet energy among functions $f\in H^1(0,\ell(t))$ with $f(0)=w(t)$ and $f(\ell(t))=0$. As a byproduct, we obtain that $u(t)=v$ in $(0,\ell(t))$ and that $u(t)$ is constant in $[\ell(t),L]$, whence $u(t)=v$ in the whole $(0,L)$ and so \eqref{eq:formu} is proved. In particular, we observe that \eqref{eq:Au1d} holds true.
        
        We now focus on \ref{def:GSL}. Assume $\ell(t)\in (0,L)$ and without loss of generality let $w(t)>0$. For an arbitrary $\rho\in (\ell(t),L)$ pick as a competitor in \ref{def:GS} the function $v=w(t)\left(1-\frac{x}{\rho}\right)^+$. Since $\{v>0\}= (0,\rho)$ and recalling that $\kappa$ vanishes in $A_0$, by using \eqref{eq:Au1d} we obtain
        \begin{align}
            \frac 12 \frac{w(t)^2}{\ell(t)}&=\frac 12 \int_0^L |\partial_x u(t)|^2\, dx\le \frac 12\int_0^{L}|\partial_x v|^2\, dx+\int_{\{v>0\}\setminus A_{u(t)}}\kappa \, dx \\
                                           & =  \frac 12 \frac{w(t)^2}{\rho}+ \int_{(0,\rho)\setminus (A_0\cup (0,\ell(t)))}\!\!\!\!\!\kappa \, dx = \frac 12 \frac{w(t)^2}{\rho}+ \int_{\ell(t)}^\rho \kappa \, dx.
        \end{align}
        Taking as a competitor $v=w(t)$, we instead deduce
        \begin{equation}
            \frac 12 \frac{w(t)^2}{\ell(t)}\le \int_{(0,L)\setminus A_{u(t)}}\kappa \, dx= \int_{\ell(t)}^L \kappa\, dx,
        \end{equation}
        namely \ref{def:GSL} is proved.

        In order to prove \ref{def:EBL} we first claim that for any differentiability point $t$ of $w$ we have
        \begin{equation}\label{eq:formdot}
            \mathfrak{h}_{A_{u(t)},\dot w(t)}=\begin{cases}
                0&\text{if }\ell(t)=0,\\
                \dot w(t)\left(1-\frac{x}{\ell(t)}\right)^+ &\text{if }\ell(t)\in(0,L),\\
                \dot w(t) &\text{if }\ell(t)=L.
            \end{cases}
        \end{equation}
        If the claim is true, then by \eqref{eq:formu} and \eqref{eq:Au1d} we readily deduce that
        \begin{subequations}\label{eq:rewriting}
        \begin{align}
            &P(t) =\int_0^L \partial_x \mathfrak{h}_{A_{u(t)},\dot w(t)}\,\partial_x u(t)\, dx &&\text{for a.e. }t\in [0,T],\\
            &F(t,\ell(t))=\frac 12 \int_0^L \partial_x u(t)^2\, dx &&\text{for all }t\in [0,T],\\
            &\int_{\ell_0}^{\ell(t)}\kappa \,dx=\int_{A_u(t)\setminus A_0}\kappa \,dx &&\text{for all }t\in [0,T],
        \end{align}
        \end{subequations}
        whence \ref{def:EBL} is just a simple rewriting of \ref{def:EB}.
        In order to prove the claim, we first observe that if $\ell(t)=0$ then \ref{def:COL} yields $w(t)=0$ and so also $\dot w (t)=0$ since $w$ is non-negative.
        Thus, the right-hand side of \eqref{eq:formdot}, for the moment denoted by $\overline u$, is really a function in $H^1_{0,\dot w(t)}((0,L), A_{u(t)})$.
        We now show that $\overline u$ fulfils \eqref{eq:euler_lagrange_dirichlet}, so we fix $\phi\in H^1(0,L)$ with $\phi(0)=0$ and vanishing outside $ A_{u(t)}=A_0\cup (0,\ell(t))$, and we compute
        \begin{equation}
            \int_0^L \partial_x \overline u \partial_x \phi\, dx=\begin{cases}
                0 &\text{if }\ell(t)=0\text{ or }\ell(t)=L,\\
                -\frac{\dot w(t)}{\ell(t)}\phi(\ell(t))  &\text{if }\ell(t)\in(0,L).
            \end{cases}
        \end{equation}
        If $\ell(t)=0$ or $\ell(t)=L$ or $\dot w(t)=0$ we conclude, so assume that $\ell(t)\in (0,L)$ and $\dot w(t)\neq 0$, which in particular implies $w(t)>0$, and let us prove that in this case there holds $\phi(\ell(t))=0$. If not, by continuity there exists a right neighbourhood of $\ell(t)$ in which $\phi$ is different from $0$, and so necessarily $A_0$ contains such neighbourhood. By choosing $\rho \in(\ell(t),L)$ in this neighbourhood, by using \ref{def:GSL} and recalling that $\kappa $ vanishes in $A_0$ we thus infer
        \begin{equation}
            \frac 12 \frac{w(t)^2}{\ell(t)}\le \frac 12 \frac{w(t)^2}{\rho}+\int_{\ell(t)}^\rho \kappa\, dx=\frac 12 \frac{w(t)^2}{\rho}<\frac 12 \frac{w(t)^2}{\ell(t)},
        \end{equation}
        and we reach a contradiction. Hence, \eqref{eq:formdot} holds true.

        We now prove the opposite implication. Properties \ref{def:CO} and \ref{def:ID} directly follow from \eqref{eq:formu} and \ref{def:COL} and \ref{def:IDL}, respectively. 
        
        We also claim that \eqref{eq:Au1d} holds true. If $\ell(t)=0$, the identity is trivial, so assume $\ell(t)>0$. If $w(t)>0$, it easily follows noting that
        \begin{equation}
            A_{u(t)}=A_0\cup\bigcup_{s\in [0,t]}\{u(s)>0\}=A_0\cup\bigcup_{\substack{{s\in [0,t]}\\{\text{s.t. } w(s)>0}}}(0,\ell(s))=A_0\cup (0,\ell(t)).
        \end{equation}
        If $w(t)=0$ instead, set $\bar t:=\min\{s\in [0,t]:\, w(s)=0\}$. If $\bar t=0$, then $w=0$ in $[0,t]$, whence $A_{u(t)}=A_0$ and
        \begin{equation}
            \int_{\ell_0}^{\ell(t)}\kappa \, dx=0
        \end{equation}
        by \ref{def:EBL}. This implies $\ell(t)=\ell_0$, and so \eqref{eq:Au1d} holds true. If $\bar t>0$, then consider a sequence $t_k$ converging to $\bar t$ from the left satisfying $w(t_k)>0$ (whence also $\ell(t_k)>0$). 
        So, we have
        \begin{equation}
            A_{u(t)}=A_0\cup\bigcup_{\substack{{s\in [0,t]}\\{\text{s.t. } w(s)>0}}}(0,\ell(s))=A_0\cup\bigcup_{k\in \n}(0,\ell(t_k))=A_0\cup (0,\ell^-(\bar t\,)).
        \end{equation}
        Moreover, from \ref{def:EBL} we deduce
        \begin{align}
            \int_{\ell^-(\bar t\,)}^{\ell(t)}\kappa\, dx&=\lim\limits_{k\to +\infty}\int_{\ell(t_k)}^{\ell(t)}\kappa\, dx=\lim\limits_{k\to +\infty}\left( F(t_k,\ell(t_k))- F(t,\ell(t))+\int_{t_k}^t P(\tau)\, d\tau\right)\\
            &=- F(t,\ell(t))+\int_{\bar t}^t P(\tau)\, d\tau+\lim\limits_{k\to +\infty} F(t_k,\ell(t_k))=\lim\limits_{k\to +\infty} F(t_k,\ell(t_k)).
        \end{align}
        If this last limit vanishes we conclude.
        Indeed, it would imply that the interval $(\ell^-(\bar t\,), \ell(t))$ is contained in $A_0$. To show it, we first observe that if $\ell(t_k)=L$ for some $k\in \n$, then $F(t_k,\ell(t_k))=0$ definitively. If instead $\ell(t_k)<L$ for all $k\in\n$, by \ref{def:GSL} we deduce
        \begin{equation}
            \lim\limits_{k\to +\infty}F(t_k,\ell(t_k))= \lim\limits_{k\to +\infty}\frac 12 \frac{w(t_k)^2}{\ell(t_k)}\le\lim\limits_{k\to +\infty}\left( \frac 12 \frac{w(t_k)^2}{\ell^-(\bar t\,)}+\int_{\ell(t_k)}^{\ell^-(\bar t\,)}\kappa \, dx\right)=0,
        \end{equation}
        whence \eqref{eq:Au1d} is proved.

        We now focus on \ref{def:GS}. Fix $t\in [0,T]$ and without loss of generality we may assume that $\ell(t)\in (0,L)$ and $w(t)>0$. We then pick $v\in H^1(0,L)$ with $v(0)=w(t)$, and we consider $\rho:=\sup\{\lambda \in (0,L]:\, (0,\lambda)\subseteq \{v>0\}\}$. If $\rho<L$, then $v(\rho)=0$ and there holds
        \begin{equation}
            \frac 12 \frac{w(t)^2}{\ell(t)}\le \frac 12 \frac{w(t)^2}{\rho}+\int_{(0,\rho)\setminus(0,\ell(t))}\kappa\, dx.
        \end{equation}
        Indeed, the above inequality is trivial in case $\rho\in (0,\ell(t)]$, while it follows by \ref{def:GSL} if $\rho\in (\ell(t),L)$. Exploiting the fact that $\kappa$ vanishes on $A_0$ we thus obtain
        \begin{align}
            \frac 12\int_0^L|\partial_x u(t)|^2\,dx&=\frac 12 \frac{w(t)^2}{\ell(t)}\le \frac 12 \frac{w(t)^2}{\rho}+\int_{(0,\rho)\setminus(0,\ell(t))}\kappa\, dx\le \frac 12\int_0^\rho|\partial_x v|^2\,dx+\int_{\{v>0\}\setminus(0,\ell(t))}\kappa\, dx\\
            &\le  \frac 12\int_0^L|\partial_x v|^2\,dx+\int_{\{v>0\}\setminus(A_0\cup (0,\ell(t)))}\kappa\, dx \\
            & = \frac 12\int_0^L|\partial_x v|^2\,dx+\int_{\{v>0\}\setminus A_{u(t)}}\kappa\, dx.
        \end{align}
        If $\rho=L$ instead, then $\{v>0\}=(0,L)$ and hence we deduce
        \begin{equation}
            \frac 12\int_0^L|\partial_x u(t)|^2\,dx=\frac 12 \frac{w(t)^2}{\ell(t)}\le \int_{\ell(t)}^L \kappa\, dx\le \frac 12\int_0^L|\partial_x v|^2\,dx+\int_{\{v>0\}\setminus A_{u(t)}}\kappa\, dx,
        \end{equation}
        so \ref{def:GS} is proved.

        We are just left to show the validity of \ref{def:EB}. By exploiting \eqref{eq:Au1d}, \ref{def:COL} and \ref{def:GSL}, arguing exactly as in the proof of the reverse implication, one deduces that $A_{u(t)}\in \mathcal O_{0,\dot w(t)}$ and that formula \eqref{eq:formdot} holds true for all differentiability points of $w$. By using \eqref{eq:rewriting}, we directly infer that the map
        \begin{equation}
            t\mapsto \int_0^L \partial_x \mathfrak{h}_{A_{u(t)},\dot w(t)}\,\partial_x u(t)\, dx
        \end{equation}
        belongs to $L^1(0,T)$ and finally \ref{def:EB} can be directly obtained from \ref{def:EBL}. 
    \end{proof}

    We finally present two examples which show how uniqueness and time-regularity of DES of the debonding model are in general not expected. Compare also with~\cite[Section 2]{RivQuas}.
    
    First, consider a constant prescribed displacement $w(t)=w>0$ and let $A_0=(0,\ell_0)$ with $\ell_0\in \left(0,\frac L2\right)$. For an arbitrary $\alpha\in \left(\ell_0,\frac L2\right]$ consider the toughness
    \begin{equation}
        \kappa(x)=\begin{cases}\displaystyle
            \frac{w^2}{2x^2} & \text{ if }x\in [\ell_0,\alpha],\vspace{2mm}
\\            \displaystyle\frac{w^2}{2\alpha^2} & \text{ if }x\in (\alpha,L].
        \end{cases}
    \end{equation}
    In this setting, we claim that any non-decreasing function $\ell\colon [0,T]\to [\ell_0,\alpha)$ with $\ell(0)=\ell_0$ gives rise to a DES via formula \eqref{eq:formu}, whence nor uniqueness nor continuity holds.

    We just need to show that \ref{def:GSL} and \ref{def:EBL} are automatically satisfied. In the current situation, notice that the global stability condition is equivalent to the system of inequalities
    \begin{equation}
        \begin{cases}\displaystyle
            \frac{1}{\ell(t)}\le \frac{1}{\rho}+\int_{\ell(t)}^\rho \frac{1}{x^2}\, dx &\text{for all }\rho\in (\ell(t),\alpha],\vspace{2mm}\\\displaystyle
            \frac{1}{\ell(t)}\le \frac{1}{\rho}+\int_{\ell(t)}^\alpha \frac{1}{x^2}\, dx+\int_{\alpha}^\rho \frac{1}{\alpha^2}\, dx&\text{for all }\rho\in (\alpha,L),\vspace{2mm}\\\displaystyle
            \frac{1}{\ell(t)}\le\int_{\ell(t)}^\alpha \frac{1}{x^2}\, dx+\int_{\alpha}^L \frac{1}{\alpha^2}\, dx.
        \end{cases}
    \end{equation}
    The first inequality is always verified, while the second one and the third one can be equivalently rewritten as
    \begin{equation}
         \begin{cases}\displaystyle
         0\le \frac{(\rho-\alpha)^2}{\rho\alpha^2}& \text{for all }\rho\in (\alpha,L),\vspace{2mm}\\\displaystyle
         0\le \frac{L-2\alpha}{\alpha^2},
         \end{cases}
    \end{equation}
    thus \ref{def:GSL} is always true.

    Observing that $P(t)\equiv 0$ since $w$ is constant, we also have
    \begin{equation}
        F(t,\ell(t))-F(0,\ell_0) +\int_{\ell_0}^{\ell(t)}\kappa\, dx= \frac{w^2}{2}\left(\frac{1}{\ell(t)}-\frac{1}{\ell_0}+\int_{\ell_0}^{\ell(t)}\frac{1}{x^2}\, dx\right)=0,
    \end{equation}
    whence also \ref{def:EBL} is fulfilled.

    A less pathological situation in which time-continuity fails is given by the following example. Let $w(t)=t$, $A_0=(0,\ell_0)$ with $\ell_0\in (0,L)$, and consider a constant toughness $\kappa>0$ in $[\ell_0,L]$. For $T>\sqrt{2\kappa}L/2$, we claim that the functions
    \begin{subequations}
    \begin{align}\label{eq:example1}
        \ell(t)&:=\begin{cases}
            \ell_0 &\text{if }t\in [0,\sqrt{2\kappa}\ell_0),\\
            \frac{t}{\sqrt{2\kappa}} &\text{if }t\in [\sqrt{2\kappa}\ell_0,\sqrt{2\kappa}L/2),\\
            L &\text{if }t\in [\sqrt{2\kappa}L/2,T],
        \end{cases}&& \text{if }\ell_0<\frac L2,\\\label{eq:example2}
        \ell(t)&:=\begin{cases}
            \ell_0 &\text{if }t\in [0,\sqrt{2\kappa\ell_0(L-\ell_0)}),\\
            L &\text{if }t\in [\sqrt{2\kappa\ell_0(L-\ell_0)},T],
        \end{cases}&& \text{if }\ell_0\ge\frac L2,
    \end{align}
    \end{subequations}
    give rise to a DES. Indeed, notice that in this setting the global stability condition is equivalent to the inequality
    \begin{equation}\label{eq:GSparticular}
        t\le \sqrt{2\kappa \ell(t)\min\{\ell(t),L-\ell(t)\}}\qquad\text{whenever }\ell(t)<L,
    \end{equation}
    while the energy balance is equivalent to the system
    \begin{equation}\label{eq:EBparticular}
        \begin{cases}\displaystyle
            \frac{t^2}{2\ell(t)}+\kappa(\ell(t)-\ell_0)=\int_0^t \frac{\tau}{\ell(\tau)}\, d\tau &\text{if }\ell(t)<L,\\\displaystyle
            \kappa(L-\ell_0)=\int_0^{t_\ell}\frac{\tau}{\ell(\tau)}\, d\tau &\text{if }\ell(t)=L,
        \end{cases}
    \end{equation}
    where $t_\ell:=\sup\{s\in[0,T]:\, \ell(s)<L\}$.

    It is now immediate to check that the debonding fronts defined in \eqref{eq:example1} and \eqref{eq:example2} fulfil both \eqref{eq:GSparticular} and \eqref{eq:EBparticular}.

\section{Properties of displacement energetic solutions}\label{sec:properties}

This section is devoted to collect some properties that displacement energetic solutions automatically possess in view of the global stability condition. We begin by showing some uniform bounds any DES satisfies.
\begin{prop}\label{prop:boundDES}
    Assume the map $t\mapsto u(t)$ satisfies \ref{def:CO} and \ref{def:GS}. Then $u(t)$ is non-negative, bounded, locally Lipschitz continuous in $\Omega$, and harmonic in $\{u(t)>0\}$ for all $t\in [0,T]$. Moreover, the following uniform-in-time bounds hold true:
    \begin{itemize}
        \item $\sup\limits_{t\in [0,T]}\|u(t)\|_{L^\infty(\Omega)}\le M$;
        \item $\sup\limits_{t\in [0,T]}\|u(t)\|_{H^1(\Omega)}\le C$;
        \item for all $\Omega'$ well contained in $\Omega$, there exists $C'>0$ such that $\sup\limits_{t\in [0,T]}\|u(t)\|_{C^{0,1}(\Omega')}\le C'$.
    \end{itemize}
    
\end{prop}
\begin{proof}
    By combining \ref{def:GS} and \Cref{lemma:minimiser_AltCaffarelli} we deduce that $u(t)$ is non-negative, bounded, locally Lipschitz continuous in $\Omega$, and harmonic in $\{u(t)>0\}$ for all $t\in [0,T]$. Moreover, by \eqref{eq:boundednessw} the following bounds hold:
    \begin{align}
        \|u(t)\|_{L^\infty(\Omega)}&\le\esssup_\Gamma w(t)\le M,\\
        \|u(t)\|_{C^{0,1}(\Omega')}&\le C'\Big(1+\esssup_\Gamma w(t)\Big)\le C'(1+M).
    \end{align}
    Above, $\Omega'$ is a well contained subset of $\Omega$.

    To obtain the uniform bound in $H^1(\Omega)$ we first use $w(t)$ as a competitor in \ref{def:GS}, deducing
    \begin{align}
        \frac 12\int_\Omega|\nabla u(t)|^2 \,dx\le \frac 12\int_\Omega|\nabla w(t)|^2 \,dx +\|\kappa\|_{L^\infty(\Omega)}|\Omega|\le \max\limits_{t\in [0,T]}\|w(t)\|^2_{H^1(\Omega)}+\|\kappa\|_{L^\infty(\Omega)}|\Omega|\le C.
    \end{align}
    We then conclude by means of Poincar\'e inequality.
\end{proof}

\begin{remark}
    The local Lipschitz continuity of $u(t)$ is essentially due to the fact that global stability \ref{def:GS} implies that $u(t)$ is an outward minimiser of the Alt--Caffarelli functional, in the terminology of~\cite{Velichkov}. On the contrary, since the inward minimality condition is not granted by \ref{def:GS}, we can not ensure that the positivity set $\{u(t)>0\}$ is of (locally) finite perimeter, unlike in~\cite{FeldKimPoz2}.
\end{remark}

We now show how the global stability condition implies an upper energy inequality, at least assuming a condition similar to \eqref{eq:safeassumption}. This result will be used later on to prove \Cref{lemma:safe}.

\begin{prop}\label{prop:upperineq}
    Let the map $t\mapsto u(t)$ satisfy \ref{def:CO} and \ref{def:GS}. Moreover,  assume that there exists $\widetilde w\in AC([0,T];H^1(\Omega))$ satisfying $\widetilde w(t)\in H^1_{\Gamma, w(t)}(\Omega, A_{u(t)})^+$ for all $t\in [0,T]$. 
    If the map
    \begin{equation}
        \displaystyle t\mapsto \int_\Omega \nabla \mathfrak{h}_{A_{u(t)},\dot w(t)}\cdot \nabla u(t) \, dx
    \end{equation}
    is measurable, then the following inequality holds true
    \begin{equation}\label{eq:upperineq}
            \frac 12 \int_\Omega |\nabla u(t)|^2 dx+\int_{A_{u(t)}\setminus A_0}\kappa\,d x\ge   \frac 12 \int_\Omega |\nabla u(0)|^2 dx+\int_0^t \int_\Omega \nabla \mathfrak{h}_{A_{u(\tau)},\dot w(\tau)}\cdot \nabla u(\tau) \, dx\, d\tau
        \end{equation}
        for all $t\in [0,T]$.
\end{prop}

\begin{proof}
     We first observe that by \ref{def:GS} and \Cref{lemma:equivalent} we know that $u(t)=\mathfrak{h}_{A_{u(t)},w(t)}$ for all $t\in [0,T]$. Moreover, by the assumptions on $\widetilde w$ it is easy to infer that $\dot{\widetilde w}(t)\in H^1_{\Gamma,\dot w(t)}(\Omega, A_{u(t)})$, whence $A_{u(t)}\in \mathcal O_{\Gamma,\dot w(t)}$, for a.e. $t\in [0,T]$. By using \eqref{eq:euler_lagrange_dirichlet}, we now obtain that
    \begin{equation}\label{eq:workeqtilde}
        \int_\Omega \nabla \dot{\widetilde w}(t)\cdot \nabla u(t) \, dx=\int_\Omega \nabla \mathfrak{h}_{A_{u(t)},\dot w(t)}\cdot \nabla u(t) \, dx \quad\text{for a.e. $t\in [0,T]$}.
    \end{equation}
    In particular, since the right-hand side above is measurable by assumption, we infer that both terms actually belong to $L^1(0,T)$: indeed, due to the uniform bound of $u(t)$ in $H^1(\Omega)$ stated in \Cref{prop:boundDES}, the left hand-side can be bounded by $\|\nabla \dot{\widetilde w}(t)\|_{L^2(\Omega)}$, up to a multiplicative constant.
    
    Let us now focus on the proof of \eqref{eq:upperineq}. If $t=0$, the inequality is true, so we fix $t\in (0,T]$, and we consider a sequence of partitions $0=s_0^n<s_1^n<\dots<s^n_{k(n)}=t$ satisfying
    \begin{align}
        &\lim\limits_{n\to +\infty}\sup_{k=1,\dots,k(n)}(s^n_k-s^n_{k-1})=0,\label{eq:1}\\
        &\lim\limits_{n\to +\infty} \sum_{k=1}^{k(n)} (s^n_k-s^n_{k-1})\int_\Omega \nabla \dot{\widetilde w}(s^n_k)\cdot \nabla u(s^n_k)\, dx= \int_0^t\int_\Omega \nabla \dot{\widetilde w}(\tau)\cdot \nabla u(\tau)\,dx\,d\tau,\label{eq:2}\\
        &\lim\limits_{n\to +\infty} \sum_{k=1}^{k(n)}\left\|(s^n_k-s^n_{k-1})\nabla \dot{\widetilde w}(s^n_k)-\int_{s^n_{k-1}}^{s^n_k} \nabla \dot{\widetilde w}(\tau)\,d\tau\right\|_{L^2(\Omega)}=0. \label{eq:3}       
    \end{align}
    Such a sequence of partitions exists by~\cite[Lemma~4.5]{FrancMielk} since $\dot{\widetilde w}\in L^1(0,T;H^1(\Omega))$ by assumption, and since
    \begin{equation}
        t\mapsto \int_\Omega \nabla\dot{\widetilde w}(t)\cdot \nabla u(t)\,dx\in L^1(0,T)
    \end{equation}
    as we previously noticed. In particular, by the absolute continuity of the integral, we observe that
    \begin{equation}\label{eq:4}
        \lim\limits_{n\to +\infty}\sup_{k=1,\dots,k(n)}\int_{s^n_{k-1}}^{s^n_k}\|\nabla \dot{\widetilde w}(\tau)\|_{L^2(\Omega)}\, d\tau=0.
    \end{equation}
 For $k=1,\dots,k(n)$ we now take as a competitor for $u(s^n_{k-1})$ in \ref{def:GS} the function $v_{k-1}^n:= u(s^n_k)-\widetilde w(s^n_{k})+\widetilde w(s^n_{k-1})$, deducing 
 \begin{equation}
        \frac 12 \int_\Omega |\nabla u(s^n_{k-1})|^2 dx\le \frac 12 \int_\Omega |\nabla u(s^n_k)-\nabla (\widetilde w(s^n_{k})-\widetilde w(s^n_{k-1}))|^2 dx+\int_{\{v^n_{k-1}>0\}\setminus A_{u(s^n_{k-1})}}\kappa\, dx.
    \end{equation}
    Observing that $\{v^n_{k-1}>0\}\setminus A_{u(s^n_{k-1})}\subseteq A_{u(s^n_{k})}\setminus A_{u(s^n_{k-1})}$ since by assumption $\widetilde w(s^n_{k-1})=0$ outside $A_{u(s^n_{k-1})}$, by summing the above inequality from $k=1$ to $k(n)$ we deduce
    {\allowdisplaybreaks\begin{align}
        &\,\frac 12 \int_\Omega |\nabla u(t)|^2 dx-\frac 12 \int_\Omega |\nabla u(0)|^2 dx+\int_{A_{u(t)}\setminus A_0}\kappa\, dx \\
        &\quad \ge \sum_{k=1}^{k(n)}\frac 12 \int_\Omega |\nabla u(s^n_k)|^2 dx-\frac 12 \int_\Omega |\nabla u(s^n_k)-\nabla (\widetilde w(s^n_{k})-\widetilde w(s^n_{k-1}))|^2 dx\\
        &\quad=\sum_{k=1}^{k(n)}\int_{s^n_{k-1}}^{s^n_k}\frac{d}{d\tau}\frac 12 \int_\Omega |\nabla u(s^n_k)-(\widetilde w(s^n_{k})-\widetilde w(\tau))|^2 dx    \, d\tau   
        \\
        &\quad=\sum_{k=1}^{k(n)}\int_{s^n_{k-1}}^{s^n_k}\int_\Omega (\nabla u(s^n_k)-\nabla \widetilde w(\tau))\big)\cdot\nabla \dot{\widetilde w}(\tau) dx\, d\tau\\
        &\quad= \sum_{k=1}^{k(n)} (s^n_k{-}s^n_{k-1})\int_\Omega \nabla \dot{\widetilde w}(s^n_k)\cdot \nabla u(s^n_k)\, dx\\
        & \quad \hphantom{=} \,\,+ \underbrace{\sum_{k=1}^{k(n)}\int_\Omega \nabla u(s^n_k)\cdot\left[\int_{s^n_{k-1}}^{s^n_k} \nabla \dot{\widetilde w}(\tau)d\tau{-}(s^n_k{-}s^n_{k-1})\nabla \dot{\widetilde w}(s^n_k)\right]dx}_{=:J_1^n}\\
        &\quad \hphantom{=} \, \,+ \underbrace{\sum_{k=1}^{k(n)}\int_{s^n_{k-1}}^{s^n_k}\int_\Omega \nabla \dot{\widetilde w}(\tau)\cdot \nabla(w(\tau)-w(s^n_k))\, dx\,d\tau}_{=:J_2^n}.
    \end{align}}
    By using \eqref{eq:2}, the first term above converges to
    \begin{equation}
        \int_0^t\int_\Omega \nabla \dot{\widetilde w}(\tau)\cdot \nabla u(\tau)\,dx\,d\tau
    \end{equation}
    as $n\to +\infty$, while we now show that both $J_1^n$ and $J_2^n$ vanish. To this aim we first estimate
    \begin{align}
        |J_1^n|\le \sum_{k=1}^{k(n)}\|\nabla u(s^n_k)\|_{L^2(\Omega)}\left\|(s^n_k-s^n_{k-1})\nabla \dot{\widetilde w}(s^n_k)-\int_{s^n_{k-1}}^{s^n_k} \nabla  \dot{\widetilde w}(\tau)\,d\tau\right\|_{L^2(\Omega)},
    \end{align}
    which goes to zero by \eqref{eq:3} since we recall that \Cref{prop:boundDES} yields the uniform bound $\|\nabla u(s^n_k)\|_{L^2(\Omega)}\le C$.

    Regarding $J_2^n$, we argue as follows:
    \begin{align}
        |J_2^n|&\le \sum_{k=1}^{k(n)}\int_{s^n_{k-1}}^{s^n_k}\|\nabla \dot{\widetilde w}(\tau)\|_{L^2(\Omega)}\int_{\tau}^{s^n_k}\|\nabla \dot{\widetilde w}(r)\|_{L^2(\Omega)}\,dr\,d\tau\le \sum_{k=1}^{k(n)}\left(\int_{s^n_{k-1}}^{s^n_k}\|\nabla \dot{\widetilde w}(\tau)\|_{L^2(\Omega)}\,d\tau\right)^2\\
        &\le \sup_{k=1,\dots,k(n)}\left(\int_{s^n_{k-1}}^{s^n_k}\|\nabla \dot{\widetilde w}(\tau)\|_{L^2(\Omega)} \, d\tau\right)\int_0^t\|\nabla \dot{\widetilde w}(\tau)\|_{L^2(\Omega)}\, d\tau,
    \end{align}
    which vanishes by \eqref{eq:4}.

    Thus, we have proved that
    \begin{equation}
            \frac 12 \int_\Omega |\nabla u(t)|^2 dx+\int_{A_{u(t)}\setminus A_0}\kappa\,d x\ge   \frac 12 \int_\Omega |\nabla u(0)|^2 dx+\int_0^t \int_\Omega \nabla \dot{\widetilde w}(\tau)\cdot \nabla u(\tau) \, dx\, d\tau,
        \end{equation}
        and so we conclude by \eqref{eq:workeqtilde}.   
    \end{proof}
    
        Observe that, given a DES $t\mapsto u(t)$, the map $t\mapsto \nabla \mathfrak{h}_{A_{u(t)},\dot w(t)}$ does not necessarily belong to $L^1(0,T;L^2(\Omega))$, although
        \begin{equation}
            \displaystyle t\mapsto \int_\Omega \nabla \mathfrak{h}_{A_{u(t)},\dot w(t)}\cdot \nabla u(t) \, dx
        \end{equation}
        is in $L^1(0,T)$ and $\nabla u$ belongs to $L^\infty(0,T;L^2(\Omega))$ by \Cref{prop:boundDES}. This is the reason why in the previous proposition we needed the existence of the function $\widetilde w$ in order to complete the argument (to be precise, we needed $\nabla\dot{\widetilde w}$ to be in $L^1(0,T;L^2(\Omega))$). We present here a counterexample even in dimension one, in the framework of \Cref{sec:1d}.

        Let $\Omega=(0,L)$, $\Gamma=\{0\}$ and let $A_0=\emptyset$ (i.e. $\ell_0=0$) and $\kappa>0$ be a positive constant. We consider $w\in AC([0,T])^+$ such that $w(0)=0$ and $\sup_{t\in [0,T]} w(t)\le \sqrt{2\kappa} L/2$ to be chosen later. It may be checked (see also~\cite[Section 2]{RivQuas}) that the unique DES of the debonding model is given by \eqref{eq:formu} with
        \begin{equation}
            \ell(t)=\frac{w^*(t)}{\sqrt{2\kappa}},
        \end{equation}
        where $w^*(t):=\sup\limits_{s\in [0,t]}w(s)$ is the smallest non-decreasing function above $w$. Without loss of generality we may assume that $w^*(t)>0$ for positive times, and so in this setting we have
        \begin{align}
            \quad\,\int_0^T\left(\int_0^L \partial_x \mathfrak{h}_{A_{u(\tau)},\dot w(\tau)}^2\, dx\right)^\frac 12d\tau & =\int_0^T\left(\int_0^L \partial_x \mathfrak{h}_{(0,\ell(\tau)),\dot w(\tau)}^2\, dx\right)^\frac 12d\tau\\
            &=\int_0^T\left(\int_0^{\ell(\tau)} \left(\frac{\dot w(\tau)}{\ell(\tau)}\right)^2\, dx\right)^\frac 12d\tau \\
            & =\int_0^T \frac{|\dot w(\tau)|}{\sqrt{\ell(\tau)}} \, d\tau=(2\kappa)^\frac 14 \int_0^T \frac{|\dot w(\tau)|}{\sqrt{w^*(\tau)}} \, d\tau.
        \end{align}
        We now choose a suitable function $w$. Consider a vanishing and strictly decreasing sequence of positive times $\{t_j\}_{j\in\n}$ such that $t_1=T$ and consider another  vanishing and strictly decreasing sequence $\{a_j\}_{j\in\n}$ satisfying $a_1\le \sqrt{2\kappa}L/2$,
        \begin{equation}
            \sum_{j=1}^\infty a_j<+\infty,\quad\text{ and }\quad\sum_{j=1}^\infty \sqrt{a_j}=+\infty.
        \end{equation}
We then define $w$ as follows (see also \Cref{fig:graph_w_l})
        \begin{equation}
            w(t):=\begin{cases}
                \frac{2 a_j}{t_j-t_{j+1}}(t-t_{j+1}) &\text{if }t\in \left(t_{j+1}, \frac{t_{j+1}+t_j}{2}\right]\text{ for some } j\in \n,\\
                \frac{2 a_j}{t_j-t_{j+1}}(t_{j}-t) &\text{if }t\in \left(\frac{t_{j+1}+t_j}{2}, t_{j} \right]\text{ for some } j\in \n,\\
                0&\text{if }t=0.
            \end{cases}
        \end{equation}

    \begin{figure}
    \centering
        \begin{tikzpicture}[scale=0.70]
            \draw[->] (0,0) -- (16,0);
            \draw[->] (0,0) -- (0,10);
    
            \node[left]  at (0, 8) {$a_1$};
    
            \node[left]  at (0, 4) {$a_2$};
    
            \node[left]  at (0, 2) {$a_3$};
    
            \node[left]  at (0, 1) {$a_4$};
    
            \node[left]  at (0, .5) {$a_5$};
    
            \node[below] at (14, 0) {$t_1$};
    
            \node[below] at (10, 0) {$t_2$};
    
            \node[below] at (6, 0) {$t_3$};
    
            \node[below] at (2, 0) {$t_4$};
    
            \node[below] at (.2, 0) {$t_5$};
    
            \draw [dashed] (0, 8) -- (14, 8);
            \draw [dashed] (14, 0) -- (14, 8);
            
            \draw [dashed] (0, 4) -- (10, 4);
            \draw [dashed] (10, 0) -- (10, 4);
    
            \draw [dashed] (0, 2) -- (6, 2);
            \draw [dashed] (6, 0) -- (6, 2);
    
            \draw [dashed] (0, 1) -- (2, 1);
            \draw [dashed] (2, 0) -- (2, 1);
    
            \draw [dashed] (.2, .5) -- (.2, 0);

            \draw [red] (0, 0) -- (.1, .5);
            \draw [red] (.1, .5) -- (.2, 0);
            \draw [red] (.2, 0) -- (1.1, 1);
            \draw [red] (1.1, 1) -- (2, 0);
            \draw [red] (2, 0) -- (4, 2);
            \draw [red] (4, 2) -- (6, 0);
            \draw [red] (6, 0) -- (8, 4);
            \draw [red] (8, 4) -- (10, 0);
            \draw [red] (10, 0) -- (12, 8);
            \draw [red] (12, 8) -- (14, 0);

            \draw [blue] (0, .25) -- (.05, .25);
            \draw [blue] (.05, .25) -- (.1, .5);
            \draw [blue] (.1, .5) -- (.65, .5);
            \draw [blue] (.65, .5) -- (1.1, 1);
            \draw [blue] (1.1, 1) -- (3, 1);
            \draw [blue] (3, 1) -- (4, 2);
            \draw [blue] (4, 2) -- (7, 2);
            \draw [blue] (7, 2) -- (8, 4);
            \draw [blue] (8, 4) -- (11, 4);
            \draw [blue] (11, 4) -- (12, 8);
            \draw [blue] (12, 8) -- (14, 8);
    
            \draw [red, ultra thick] (16, 10) -- (17, 10);
            \node [right] at (17, 10) {$ w(t) $};
            \draw [blue, ultra thick] (16, 9.2) -- (17, 9.2);
            \node [right] at (17, 9.2) {$ w^\ast(t) $};

        \end{tikzpicture}
        \caption{The graphs of $w(t)$ and $w^\ast(t)$.}
        \label{fig:graph_w_l}
    \end{figure}
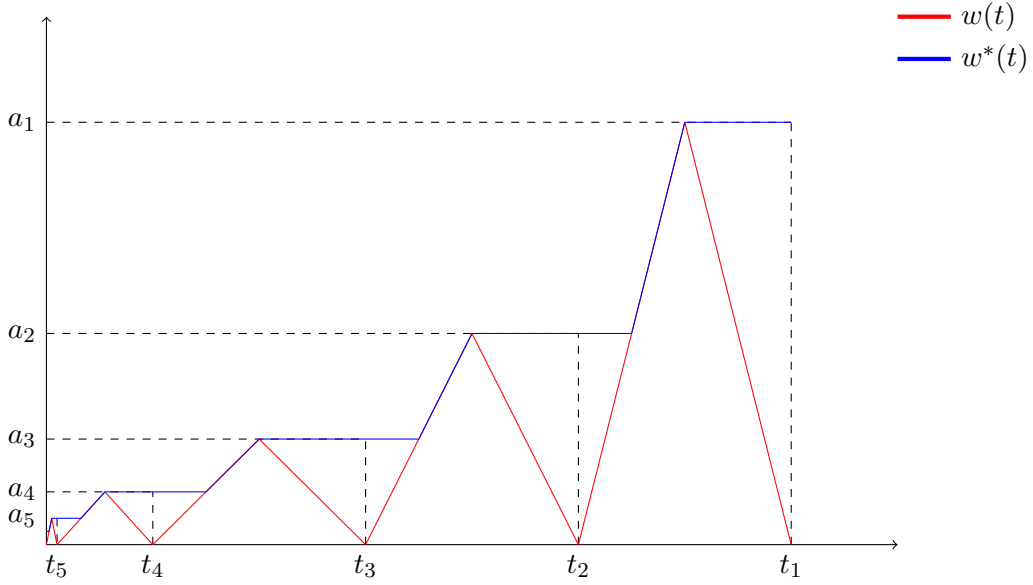
    
    It is then immediate to check that
    \begin{equation}
        \int_0^T |\dot w(\tau)|\, d\tau=\sum_{j=1}^\infty  \int_{t_{j+1}}^{t_j} \frac{2 a_j}{t_j-t_{j+1}}\, d\tau=2\sum_{j=1}^\infty a_j<+\infty,
    \end{equation}
    whence $w$ is absolutely continuous, but
    \begin{align}
        \int_0^T \frac{|\dot w(\tau)|}{\sqrt{w^*(\tau)}} \, d\tau&=\sum_{j=1}^\infty \frac{2 a_j}{t_j-t_{j+1}} \int_{t_{j+1}}^{t_j}\frac{1}{\sqrt{w^*(\tau)}} \, d\tau\ge \sum_{j=1}^\infty \frac{2 a_j}{t_j-t_{j+1}} \int_{t_{j+1}}^{t_j}\frac{1}{\sqrt{a_j}} \, d\tau\\
        &=2\sum_{j=1}^\infty \sqrt{a_j}=+\infty,
    \end{align}
    and so $t\mapsto \partial_x \mathfrak{h}_{A_{u(t)},\dot w(t)}$ does not belong to $L^1(0,T;L^2(0,L))$.

    \begin{remark}
    In the previous example, the underlying issue is given by the fact that at the initial time the film is completely glued to the substrate, while the regularity of $w$ does not play any role. Indeed, by properly tuning the sequence $\{t_j\}_{j\in\n}$, the function $w$ may be constructed to be Lipschitz continuous (or even smooth after a suitable regularization).
    \end{remark}

    We now make use of \Cref{prop:upperineq} in order to show the validity of \Cref{lemma:safe}.
  
    \begin{proof}[Proof of \Cref{lemma:safe}]
     We already know by \Cref{prop: almostEQ} that, given a SES $t\mapsto A(t)$, the map $u(t):=\mathfrak{h}_{A(t),w(t)}$ satisfies \ref{def:CO}, \ref{def:ID}, and \ref{def:GS}. 
     Moreover, inclusion \eqref{eq:contained} is true and $A_{u(0)}=A_0$. In particular, we deduce that $A_0\subseteq A_{u(t)}$ for all $t\in [0,T]$.

     Assumption \eqref{eq:safeassumption} now ensures that $\overline w\in AC([0,T];H^1(\Omega))$ satisfies $\overline w(t)\in H^1_{\Gamma,w(t)}(\Omega,A_{u(t)})^+$ for all $t\in [0,T]$. Furthermore, notice that $\dot{\overline w}(t)\in H^1_{\Gamma,\dot w(t)}(\Omega, A_{u(t)})$, and so that $A_{u(t)}\in \mathcal O_{\Gamma,\dot w(t)}$, for all common differentiability points of $w$ and $\overline w$. By \eqref{eq:contained} and \eqref{eq:euler_lagrange_dirichlet} we thus infer that
    \begin{equation}\label{eq:workidentity}
         \int_\Omega \nabla {\mathfrak{h}_{A(t),\dot w(t)}}\cdot \nabla \mathfrak{h}_{A(t),w(t)} \, dx=\int_\Omega \nabla \mathfrak{h}_{A_{u(t)},\dot w(t)}\cdot \nabla u(t) \, dx,\qquad \text{for a.e. }t\in [0,T].
    \end{equation}
    In particular, by \ref{def:EBS}, the right-hand side above belongs to $L^1(0,T)$. We are now in a position to exploit \Cref{prop:upperineq}: by combining it with \ref{def:EBS}, \eqref{eq:contained}, and \eqref{eq:workidentity} we indeed deduce
    \begin{align}
        \frac 12 \int_\Omega|\nabla u(t)|^2 \, dx+\int_{A_{u(t)}\setminus A_0}\kappa\,d x&\ge  \frac 12 \int_\Omega|\nabla u(0)|^2 \, dx+\int_0^t\int_\Omega \nabla \mathfrak{h}_{A_{u(\tau)},\dot w(\tau)}\cdot \nabla u(\tau) \, dxd\tau\\
        &=\mathcal E(0,A_0)+\int_0^t \int_\Omega \nabla {\mathfrak{h}_{A(\tau),\dot w(\tau)}}\cdot \nabla \mathfrak{h}_{A(\tau),w(\tau)} \, dxd\tau\\
        &= \mathcal E(t,A(t))+\int_{A(t)\setminus A_0}\kappa\,d x=\frac 12 \int_\Omega|\nabla u(t)|^2 \, dx+ \int_{A(t)\setminus A_0}\kappa\,d x\\
        &\ge \frac 12 \int_\Omega|\nabla u(t)|^2 \, dx+ \int_{A_{u(t)}\setminus A_0}\kappa\,d x.
    \end{align}
    
    In particular, we have proved that
    \begin{equation}
        \int_{A_{u(t)}\setminus A_0}\kappa\,d x=\int_{A(t)\setminus A_0}\kappa\,d x,\qquad\text{for all }t\in [0,T].
    \end{equation}
    Since $\kappa$ is positive outside $A_0$, this yields that $A_{u(t)}\setminus A_0=A(t)\setminus A_0$, whence \eqref{eq:hypinclusion} holds true and we conclude.
    \end{proof}

\section{Construction of displacement energetic solutions}\label{sec:existence}

In this section we prove \Cref{thm:safe} by constructing a DES in the sense of \Cref{defi:DES}, requiring assumption \eqref{eq:safeassumption}. To do so, we follow the classical approach of the Minimizing Movement scheme. For $j \in \n$ we consider the step-size $\tau_j:=\frac{T}{j}$ and for $i = 0, \dots, j$ we define the discrete times $t_i^j:=i \tau_j$.

Starting from the initial pair
\begin{equation}
    u_0^j:= \mathfrak{h}_{A_0,w(0)}\qquad \text{ and }\qquad A_0^j:= A_0,
\end{equation}
for $i=1,...,j$ we construct recursively the functions $u_i^j$ and the sets $A_i^j$ via the following algorithm:
\begin{subnumcases}{}
   \displaystyle u_i^j\in \operatorname*{arg\ min}\limits_{v\in H^1_{\Gamma,w(t^j_i)}(\Omega)}\mathcal{AC}(v,A_{i-1}^j),\label{eq:MMS1}\\
    A^j_i:=\{u_i^j>0\}\cup A^j_{i-1},\label{eq:MMS2}
\end{subnumcases}
where the Alt--Caffarelli functional $\mathcal{AC}$ has been introduced in \eqref{eq:E}. The existence of the minimisers $u^j_i$ is granted by \Cref{lemma:minimiser_AltCaffarelli}, which also implies that the sets $A_i^j$ are open. 

We now introduce the piecewise-constant interpolant $u^j$ as
\begin{equation}\label{piece}
    u^j(t) := \begin{cases}
        u_{i}^j & \text{if } t \in (t_i^j, t_{i+1}^j] \text{ for some $i\in\{1,...,j\}$}\,,\\
        \mathfrak{h}_{A_0,w(0)}&\text{if }t=0.
    \end{cases}
\end{equation}
In the same way we also define $A^j(t)$, and we set
\begin{equation}
    \overline w^j(t):= \begin{cases}
        \overline{w}(t_{i}^j) & \text{if } t \in (t_i^j, t_{i+1}^j] \text{ for some $i\in\{1,...,j\}$}\,,\\
        \overline w(0)&\text{if }t=0.
    \end{cases}
\end{equation}
Observe that, by construction, the set-valued function $t\mapsto A^j(t)$ is non-decreasing with respect to set inclusion.

\begin{remark}\label{remark:convergence_wj}
    Since $\overline w \in \AC([0, T]; H^1(\Omega))$ by assumption, for any sequence $t_n \to t$ we have $\overline w(t_n) \to \overline w(t)$ strongly in $H^1(\Omega)$. In particular, $\overline w^j(t) \to \overline w(t)$ strongly in $H^1(\Omega)$ uniformly in $[0, T]$.
\end{remark}

Similarly to \Cref{prop:boundDES}, the piecewise constant displacement fulfils the following properties and uniform bounds.

\begin{prop}\label{prop:bound} For all $t\in[0,T]$ the functions $u^j(t)$ are non-negative, bounded and locally Lipschitz continuous in $\Omega$. Moreover, $u^j(t)$ is harmonic in $\{u^j(t)>0\}$. Finally, the following uniform bounds hold true
\begin{enumerate}[(i)]
    \item $\sup\limits_{t\in [0,T]}\|u^j(t)\|_{L^\infty(\Omega)}\le M$;
    \item $\sup\limits_{t\in [0,T]}\|u^j(t)\|_{H^1(\Omega)}\le C$;
    \item for all $\Omega$ well contained in $\Omega$, there exists $C'>0$ such that $\sup\limits_{t\in [0,T]}\|u^j(t)\|_{C^{0,1}(\Omega')}\le C'$.
\end{enumerate}
\end{prop}

\begin{proof}
By construction, $u^j(t)$ minimizes $\mathcal{AC}(\cdot ,A^j(t-\tau_j))$ in $H^1_{\Gamma,w^j(t)}(\Omega)$, so the proof can be carried out analogously to \Cref{prop:boundDES}.
\end{proof}

We will now make use of the following compactness result of Helly's type, stated and proved in~\cite[Lemma 6.1]{RivScilSol}, in order to extract convergent subsequences of $t\mapsto \chi_{A^{j}(t)}$.

\begin{lemma}\label{helly}
	Let $X,Y$ be Banach spaces such that $X$ is reflexive and separable and $X$ continuously and densely embeds in $Y$. Then, given a sequence of functions $f_n\colon [a,b]\to X$ such that for a positive constant $C>0$ independent of $n$ it holds
	\begin{equation}
		\sup\limits_{t\in [a,b]}\|f_n(t)\|_X\le C\quad\text{ and }\quad {\rm Var}_Y(f_n;a,b)\le C,
	\end{equation}
 	there exist a function $f\colon [a,b]\to X$, bounded in $X$ and with bounded variation in $Y$, and a (non-relabelled) subsequence such that
\begin{equation}\label{weakconvhelly}
	f_n(t)\xrightharpoonup[n\to +\infty]{}f(t),\quad\text{weakly in $X$ for every }t\in [a,b].
\end{equation}
\end{lemma}
\begin{prop}\label{prop:rho}
    There exists a (non-relabelled) subsequence and a map $\rho\colon [0,T]\times\Omega\to \r$ such that 
    \begin{equation}
        \chi_{A^{j}(t)} \xrightharpoonup[j\to +\infty]{} \rho(t) \quad \text{weakly$^\ast$ in } L^\infty(\Omega) \quad \text{for all } \, t \in [0, T] \,.
    \end{equation}
    In particular, $\rho$ is non-decreasing, $0\leq \rho(t)\leq 1$ for all $t \in [0, T]$ and $\rho(0)=\chi_{A_0}.$
\end{prop}

\begin{proof}
    The functions $\chi_{A^j(t)}$ are clearly bounded in $L^2(\Omega)$, as they take values in $[0,1]$. Moreover, since $t\mapsto \chi_{A^j(t)}$ is monotone, it is also bounded in $\BV([0,T];L^1(\Omega))$. Using \Cref{helly}, together with a diagonal argument, we conclude that, along a non-relabelled subsequence,
    \begin{equation}
        \chi_{A^{j}(t)} \xrightharpoonup[j\to +\infty]{} \rho(t) \quad \text{weakly in } L^2(\Omega) \quad \text{for all } \, t \in [0, T] \,.
    \end{equation}
    For any $t\in[0,T],\,\chi_{A^j(t)}$ is also bounded in $L^\infty(\Omega)$. Thus, from any subsequence we can extract a further subsequence that weakly$^*$ converges in $L^\infty(\Omega)$.
    Since the limit is uniquely identified and coincides with $ \rho(t) $, the desired convergence holds without a further extraction.
\end{proof}
\begin{remark}
    Since $\rho(t)$ satisfies $1\geq \rho(t)\geq \rho(0)=\chi_{A_0}$, it follows that $\rho(t)=1$ a.e. in $A_0$ for all $t\in [0,T]$.
\end{remark}

The subsequence given in \Cref{prop:rho} now identifies a set-valued map $t \mapsto A(t)$, defined as
\begin{equation}\label{eq:definition_esssupp}
        A(t) := \Omega\setminus \operatorname{ess\ supp} (1-\rho(t))=\bigcup\{B \subseteq \Omega \colon B \text{ is open and } \rho(t) = 1 \text{ a.e. in } B\}\,.
\end{equation}
Note that, by construction, $A(t)$ is open, $A(0)=A_0$, $\rho(t) = 1$ a.e. on $A(t)$ and $t \mapsto A(t)$ is non-decreasing with respect to set inclusion.

We now show that the discrete displacements $u^j$ converge along the same subsequence to the function $\mathfrak{h}_{A(\cdot),w(\cdot)}$.
This will be done by proving that the limit function is globally stable, and exploiting \Cref{lemma:equivalent}.

\begin{prop}\label{teo:convergence_u}
    There exists a map $u\colon[0,T]\to H^1(\Omega)$ such that, for all $t\in [0,T]$, there holds
    \begin{equation}\label{eq:convergence}
        u^{j}(t) \xrightarrow[j\to +\infty]{} u(t) \quad \text{strongly in $H^1(\Omega)$, weakly$^*$ in $L^\infty(\Omega)$ and locally uniformly in $\Omega$,}
    \end{equation} 
    for the same subsequence given by \Cref{prop:rho}. Moreover, $u$ satisfies \ref{def:GS} and
    \begin{equation}\label{eq:u=uA}
        u(t)=\mathfrak{h}_{A(t),w(t)}\qquad\text{for all }t\in [0,T],
    \end{equation}
    where $A(t)$ is the set introduced in \eqref{eq:definition_esssupp}. In particular, \ref{def:CO} and \ref{def:ID} are fulfilled and $u(t)$ belongs to $H^1_{\Gamma,w(t)}(\Omega, A(t))^+\cap L^\infty(\Omega)\cap C^{0,1}_{\rm loc}(\Omega)$ for all $t\in [0,T]$. 
\end{prop}

\begin{proof}
Fix $t\in [0,T]$. By \Cref{prop:bound} we deduce that, up to taking a further subsequence $j_n$ (possibly depending on $t$), the convergence stated in \eqref{eq:convergence} holds true for a function $u(t)\in H^1_{\Gamma,w(t)}(\Omega)^+\cap L^\infty(\Omega)\cap C^{0,1}_{\rm loc}(\Omega)$, except that the convergence in $H^1(\Omega)$ is just weak. We now aim at showing the characterization \eqref{eq:u=uA}, which would yield \ref{def:ID} and the validity of \eqref{eq:convergence} without taking any subsequence.

Let us first prove that
\begin{equation}\label{eq:claimcontained}
    \{u(t)>0\}\subseteq A(t).
\end{equation}
Fix $\bar x\in\{u(t)>0\}$, so that by continuity of $u(t)$ and by the locally uniform convergence of $u^{j_n}(t)$ to $u(t)$ there exists a radius $r>0$ such that for sufficiently large $n$ one has $u^{j_n}(t,x)\ge u(t,\bar x)/2$ for all $x\in B_r(\bar x)$, whence $B_r(\bar x)\subseteq \{u^{j_n}(t)>0\}\subseteq A^{j_n}(t)$. This implies that
\begin{equation}
    \int_{B_r(\bar x)}\rho(t)\, dx=\lim\limits_{n\to +\infty}\int_\Omega \chi_{B_r(\bar x)} \chi_{A^{j_n}(t)}\, dx= \lim\limits_{n\to +\infty}\int_\Omega \chi_{B_r(\bar x)} \, dx=|B_r(\bar x)|,
\end{equation}
and so $\rho(t)=1$ a.e. in $B_r(\bar x)$. By definition \eqref{eq:definition_esssupp} of $A(t)$ we thus obtain that $B_r(\bar x)\subseteq A(t)$ and so \eqref{eq:claimcontained} is proved. In particular, we deduce that $u(t)\in H^1_{\Gamma,w(t)}(\Omega, A(t))^+$, and that
\begin{equation}\label{eq:AuinA}
    A_{u(t)}=A_0\cup\bigcup_{s\in [0,t]}\{u(s)>0\}\subseteq A_0\cup\bigcup_{s\in [0,t]}A(s)=A(t).
\end{equation}

In order to prove \ref{def:GS}, we now fix $v\in H^1_{\Gamma,w(t)}(\Omega)^+\cap C^0(\Omega)$ such that $A(t)\subseteq \{v>0\}$, and we consider the competitor $v^n(t):=v-\overline w(t)+\overline w^{j_n}(t)\in H^1_{\Gamma,w^{j_n}(t)}(\Omega)$. We recall that by \Cref{remark:convergence_wj} we have
\begin{equation}\label{eq:strongconv}
    v^n(t)\to v,\quad \text{strongly in }H^1(\Omega).
\end{equation}
By minimality of $u^j_i$ in \eqref{eq:MMS1} and by definition \eqref{eq:MMS2} of $A^j_i$ we thus infer
\begin{align}
    & \frac 12 \int_\Omega |\nabla u^{j_n}(t)|^2 \, dx+\int_{A^{j_n}(t)\setminus A^{j_n}(t-\tau_{j_n})} \kappa \, dx\\
    &\quad = \frac 12 \int_\Omega |\nabla u^{j_n}(t)|^2 \, dx+\int_{\{u^{j_n}(t)>0\}\setminus A^{j_n}(t-\tau_{j_n})} \kappa \, dx\\
    &\quad \le \frac 12 \int_\Omega |\nabla v^{n}(t)|^2 \, dx+\int_{\{v^{n}(t)>0\}\setminus A^{j_n}(t-\tau_{j_n})} \kappa \, dx= \frac 12 \int_\Omega |\nabla v^{n}(t)|^2 \, dx+\int_{\{v>0\}\setminus A^{j_n}(t-\tau_{j_n})} \kappa \, dx,
\end{align}
where in the last equality we exploited the identity
\begin{equation}
    \{v^{n}(t)>0\}\setminus A^{j_n}(t-\tau_{j_n})=\{v>0\}\setminus A^{j_n}(t-\tau_{j_n})
\end{equation}
which follows from \eqref{eq:safeassumption}, since $v^n(t)=v$ outside $A_0$, ando so in particular outside $A^{j_n}(t-\tau_{j_n})$.

By adding the term
\begin{equation}
    \int_{A^{j_n}(t-\tau_{j_n})} \kappa \,d x
\end{equation}
to both sides above and using the trivial identity $(A\setminus B)\cup B= (B\setminus A)\cup A=A\cup B$ we finally obtain
\begin{align}
     \frac 12 \int_\Omega |\nabla u^{j_n}(t)|^2 \, dx+\int_{A^{j_n}(t)} \kappa \, dx&\le \frac 12 \int_\Omega |\nabla v^{n}(t)|^2 \, dx+\int_{\{v>0\}} \kappa \, dx+ \int_{ A^{j_n}(t-\tau_{j_n})\setminus\{v>0\}} \kappa \, dx\\
     &\le \frac 12 \int_\Omega |\nabla v^{n}(t)|^2 \, dx+\int_{\{v>0\}} \kappa \, dx+ \int_{ A^{j_n}(t)\setminus A(t)} \kappa \, dx,
\end{align}
where in the last inequality we exploited the inclusions $ A^{j_n}(t-\tau_{j_n})\subseteq  A^{j_n}(t)$ and $A(t)\subseteq \{v>0\}$.

Sending $n\to +\infty$, by weak lower semicontinuity of the Dirichlet functional and exploiting \eqref{eq:strongconv} and \Cref{prop:rho} we now infer
\begin{align}
    \frac 12 \int_\Omega |\nabla u(t)|^2 \, dx+\int_\Omega \kappa \rho(t)\, dx&\le \frac 12 \int_\Omega |\nabla v|^2 \, dx+\int_{\{v>0\}} \kappa \, dx+\int_\Omega \kappa \rho(t)(1-\chi_{A(t)})\, dx\\
    &=\frac 12 \int_\Omega |\nabla v|^2 \, dx+\int_{\{v>0\}\setminus A(t)} \kappa \, dx+\int_\Omega \kappa \rho(t)\, dx,
\end{align}
where we used the equality $\rho(t)=1$ a.e. on $A(t)$. We have thus proved the validity of \eqref{item:equiv4} in \Cref{lemma:equivalent}, which in particular yields \eqref{eq:u=uA} and also \ref{def:GS} recalling \eqref{eq:AuinA}.

We are just left to prove that $u^{j}\to u(t)$ strongly in $H^1(\Omega)$.
Arguing similarly as before, but taking as a competitor for $u_i^j$ the function $v^j(t)=u(t)-\overline w(t)+\overline w^{j}(t)\in H^1_{\Gamma,w^{j}(t)}(\Omega)$, we deduce
\begin{equation}
    \frac 12 \int_\Omega |\nabla u^{j}(t)|^2 \, dx+\int_{A^{j}(t)} \kappa \, dx\le \frac 12 \int_\Omega |\nabla v^{j}(t)|^2 \, dx+\int_{\{u(t)>0\}} \kappa \, dx+ \int_{ A^{j}(t)\setminus\{u(t)>0\}} \kappa \, dx.
\end{equation}
Sending $j\to +\infty$, exploiting \eqref{eq:claimcontained} together with the fact that $\rho(t)=1$ a.e. in $A(t)$ we thus obtain
\begin{align}
   &\limsup\limits_{j\to +\infty}\frac 12 \int_\Omega |\nabla u^{j}(t)|^2 \, dx+\int_\Omega \kappa\rho(t) \, dx\\
   &\quad \le \frac 12 \int_\Omega |\nabla u(t)|^2 \, dx+\int_{\{u(t)>0\}} \kappa \, dx+ \int_\Omega \kappa\rho(t) (1-\chi_{\{u(t)>0\}}) \, dx\\
   &\quad =\frac 12 \int_\Omega |\nabla u(t)|^2 \, dx+\int_\Omega \kappa\rho(t) \, dx,
\end{align}
and so we conclude.
\end{proof}

Due to \Cref{prop:upperineq}, in order to conclude the proof of \Cref{thm:safe} we just need to show that $u$ satisfies the opposite inequality of \eqref{eq:upperineq}. This will be done in the following proposition. Indeed, notice that assumption \eqref{eq:safeassumption} automatically yields that $A_{u(t)}\in \mathcal O_{\Gamma,\dot w(t)}$ for a.e. $t\in [0,T]$   and that
    \begin{equation}
        \int_\Omega \nabla \dot{\overline w}(t)\cdot \nabla u(t) \, dx=\int_\Omega \nabla \mathfrak{h}_{A_{u(t)},\dot w(t)}\cdot \nabla u(t) \, dx \quad\text{for a.e. $t\in [0,T]$}.
    \end{equation}
Also, observe that the left-hand side above is measurable since it is the pointwise countable limit of measurable functions.

\begin{prop}
    The limit function $u$ obtained in \Cref{teo:convergence_u} satisfies the inequality
    \begin{equation}
            \frac 12 \int_\Omega |\nabla u(t)|^2 dx+\int_{A_{u(t)}\setminus A_0}\kappa\,d x\le   \frac 12 \int_\Omega |\nabla u(0)|^2 dx+\int_0^t \int_\Omega \nabla \dot {\overline w}(\tau)\cdot \nabla u(\tau) \, dx\, d\tau
    \end{equation}
    for all $t\in [0,T]$.
\end{prop}
\begin{proof}
    If $t = 0$, the inequality is trivial since $A_{u(0)}=A_0$. So fix $t \in (0, T]$ and for $j\in \n$ let $I_t^j\in\{1,\dots, j\}$ satisfy $u^j(t)=u^j_{I_t^j}$. For all $i\in \{1,\dots, I^j_t\}$, by taking as a competitor for $u^j_i$ the function $v=u^j_{i-1}+\overline w(t^j_i)-\overline w(t^j_{i-1})\in H^1_{\Gamma,w(t^j_i)}(\Omega)$ we deduce
    \begin{align}
        \frac{1}{2}\int_\Omega |\nabla u^j_{i}|^2 \, dx + \int_{A^j_i\setminus A^j_{i-1}}\kappa\,dx&=\frac{1}{2}\int_\Omega |\nabla u^j_{i}|^2 \, dx + \int_{\{u^j_i>0\}\setminus A^j_{i-1}}\kappa\,dx \\
        &\leq \frac{1}{2}\int_\Omega |\nabla u^j_{i-1}+\nabla (\overline w(t^j_i)-\overline w(t^j_{i-1}))|^2 \, dx + \underbrace{\int_{\{u^j_{i-1}>0\}\setminus A^j_{i-1}}\kappa \,dx}_{=0},
    \end{align}
    where in the last inequality we exploited the assumption $\overline w(t)=0$ outside $A_0$.

    Subtracting
    \begin{equation}
        \frac 12 \int_\Omega |\nabla u^j_{i-1}|^2 \, dx
    \end{equation}
    from both sides and summing from $i=1, \dots, I^j_t$ we thus obtain
    \begin{align}
        &\frac 12 \int_\Omega |\nabla u^j(t)|^2 dx-\frac 12 \int_\Omega |\nabla u(0)|^2 dx+\int_{A^j(t)\setminus A_0}\kappa\,d x\\
        & \quad \le \int_0^{t^j_{I^j_t}} \,\int_\Omega (\nabla u^j(\tau)+\nabla (\overline w(\tau)-\overline w^j(\tau)))\cdot \nabla \dot{\overline w}(\tau) dxd\tau\\
        &\quad = \int_0^t \int_\Omega \nabla u^j(\tau)\cdot \nabla \dot{\overline w}(\tau)\, dxd\tau-\underbrace{\int_{t^j_{I^j_t}}^t \int_\Omega \nabla u^j(\tau)\cdot \nabla \dot{\overline w}(\tau)\, dxd\tau}_{=: R_1^j}\\
        &\quad \hphantom{=}\,\, +\underbrace{\int_0^{t^j_{I^j_t}} \,\int_\Omega \nabla (\overline w(\tau)-\overline w^j(\tau))\cdot \nabla \dot{\overline w}(\tau) dxd\tau}_{=: R_2^j}.
    \end{align}
Let us first show that both $R_1^j$ and $R_2^j$ vanish as $j\to +\infty$. It is a byproduct of the following estimates, recalling \Cref{remark:convergence_wj}, the bounds of \Cref{prop:bound} and that $\lim\limits_{j\to +\infty} t^j_{I^j_t}=t$.
\begin{align}
    |R_1^j| & \le \int_{t^j_{I^j_t}}^t \|\nabla u^j(\tau)\|_{L^2(\Omega)} \|\nabla \dot{\overline w}(\tau)\|_{L^2(\Omega)} \, d\tau\le \sup\limits_{s\in [0,T]}\| u^j(s)\|_{H^1(\Omega)}\int_{t^j_{I^j_t}}^t \| \dot{\overline w}(\tau)\|_{H^1(\Omega)}\, d\tau;\\
    |R_2^j| & \le\! \int_0^t \!\! \|\nabla (\overline w(\tau){-}\overline w^j(\tau))\|_{L^2(\Omega)}\|\nabla \dot{\overline w}(\tau)\|_{L^2(\Omega)}\, d\tau\le \!\sup\limits_{s\in [0,T]}\| \overline w(s){-}\overline w^j(s)\|_{H^1(\Omega)}\!\int_0^t \!\!\| \dot{\overline w}(\tau)\|_{H^1(\Omega)}\, d\tau.
\end{align}
Sending $j\to +\infty$ we thus obtain
\begin{align}
    \frac 12 \int_\Omega |\nabla u(t)|^2 dx-\frac 12 \int_\Omega |\nabla u(0)|^2 dx+\int_{\Omega}\kappa\rho(t)(1-\chi_{A_0})\,d x\le \int_0^t \int_\Omega  \nabla \dot{\overline w}(\tau)\cdot \nabla u(\tau)\, dxd\tau. 
\end{align}
Observing that, using \eqref{eq:AuinA}, we have
\begin{equation}
    \int_{\Omega}\kappa\rho(t)(1-\chi_{A_0})\,d x\ge \int_{A(t)}\kappa\rho(t)(1-\chi_{A_0})\,d x=\int_{A(t)\setminus A_0}\kappa\,d x\ge \int_{A_{u(t)}\setminus A_0}\kappa\,d x,
\end{equation}
we finally conclude.            
\end{proof}

\subsection{Proof of \Cref{thm:general}}
For $\varepsilon>0$, define the tubular neighbourhood
\begin{equation}
    \Gamma^\varepsilon:=\{x\in \Omega :\, \operatorname{dist}(x;\Gamma)<\varepsilon\}
\end{equation}
and set $A_0^\varepsilon:= A_0\cup \Gamma^\varepsilon$. Since $\Gamma^\varepsilon$ is open, there exists a cut-off function $\Phi^\varepsilon\in C^1(\overline\Omega)$ such that $0\le \Phi^\varepsilon\le 1$ in $\overline{\Omega}$, $\Phi^\varepsilon=1$ in $\Gamma$, $\Phi^\varepsilon=0$ in $\Omega\setminus \Gamma^\varepsilon$ and 
\begin{equation}\label{eq:nablaphi}
    |\nabla \Phi^\varepsilon(x)|\le \frac{C}{\varepsilon}\qquad\text{for all }x\in \overline{\Omega}.
\end{equation}
Thus, the function $w^\varepsilon(t,x):=\Phi^\varepsilon(x) w(t,x)$ satisfies \eqref{eq:safeassumption} with $A_0^\varepsilon$ in place of $A_0$, and so by \Cref{thm:safe} and \Cref{rmk:nostable} (indeed, $A_0^\varepsilon$ may be not stable) there exists a map $t\mapsto u^\varepsilon(t)$ which satisfies:
\begin{enumerate}[leftmargin = !, labelwidth=1.2cm, align = left]
        \item [{\crtcrossreflabel{\textup{(CO)}$^\varepsilon$}[def:COeps]}] $u^\varepsilon(t)\in H^1_{\Gamma,w(t)}(\Omega)\cap C^0(\Omega)$ for all $t\in (0,T]$;
        \item[{\crtcrossreflabel{\textup{(ID)}$^\varepsilon$}[def:IDeps]}] $u^\varepsilon(0)=\mathfrak{h}_{A^\varepsilon_0,w(0)}$;
        \item[{\crtcrossreflabel{\textup{(GS)}$^\varepsilon$}[def:GSeps]}] for all $t\in (0,T]$ there holds
        \begin{equation}
            \frac 12 \int_\Omega |\nabla u^\varepsilon(t)|^2 dx\le \frac 12 \int_\Omega |\nabla v|^2 dx+\int_{\{v>0\}\setminus A_{u^\varepsilon(t)}}\kappa\, dx,\quad \text{for all }v\in H^1_{\Gamma,w(t)}(\Omega);
        \end{equation}
        \item[{\crtcrossreflabel{\textup{(EI)}$^\varepsilon$}[def:EIeps]}] for all $t\in [0,T]$ there holds
        \begin{equation}
            \frac 12 \int_\Omega |\nabla u^\varepsilon(t)|^2 dx+\int_{A_{u^\varepsilon(t)}\setminus A^\varepsilon_{0}}\kappa\,d x \le   \frac 12 \int_\Omega |\nabla u^\varepsilon(0)|^2 dx+\int_0^t \int_\Omega \nabla \dot w^\varepsilon(\tau)\cdot \nabla u^\varepsilon(\tau) \, dx\, d\tau.
        \end{equation}
    \end{enumerate}
    From \Cref{lemma:equivalent}, we also recall that $u^\varepsilon(t)=\mathfrak{h}_{A_{u^\varepsilon(t)},w(t)}$ for all $t\in [0,T]$.

    By exploiting again \Cref{helly}, as already done in \Cref{prop:rho}, we directly infer the existence of a (non relabelled) subsequence and a map $\sigma\colon [0,T]\times\Omega\to \r$ such that 
    \begin{equation} \label{eq.sigma}        \chi_{{A_{u^\varepsilon(t)}}} \xrightharpoonup[\varepsilon\to 0]{} \sigma(t) \quad \text{ weakly$^*$ in } L^\infty(\Omega) \quad \text{for all } \, t \in [0, T] \,.    \end{equation}
    In particular, $\sigma$ is non-decreasing, $0\leq \sigma(t)\leq 1$ for all $t \in [0, T]$ and $\sigma(0)=\chi_{A_0}.$ Analogously to \eqref{eq:definition_esssupp}, we now introduce the set
\begin{equation}\label{eq:Asigma}
        A(t) := \Omega\setminus \operatorname{ess\ supp} (1-\sigma(t)),
\end{equation}
so that $t\mapsto A(t)$ is non-decreasing, $A(0)=A_0$ and for all $t\in [0,T]$ the set $A(t)$ is open and $\sigma(t)=1$ a.e. in $A(t)$. The following proposition concludes the proof of Theorem \ref{thm:general}.

\begin{prop}\label{teo:convergence_ueps}
    There exists a map $u\colon[0,T]\to H^1(\Omega)$ such that for $t=0$ there holds
     \begin{equation}\label{eq:convergenceeps0}
        u^{\varepsilon}(0) \xrightarrow[\varepsilon\to 0]{} u(0) \quad \text{strongly in $H^1(\Omega)$ and weakly$^*$ in $L^\infty(\Omega)$,}
    \end{equation}
    while for $t\in (0,T]$ one has
    \begin{equation}\label{eq:convergenceeps}
        u^{\varepsilon}(t) \xrightarrow[\varepsilon\to 0]{} u(t) \quad \text{strongly in $H^1(\Omega)$, weakly$^*$ in $L^\infty(\Omega)$ and locally uniformly in $\Omega$,}
    \end{equation} 
    for the same subsequence satisfying \eqref{eq.sigma}. Moreover, $u$ satisfies \ref{def:GS} and
    \begin{equation}\label{eq:u=uAeps}
        u(t)=\mathfrak{h}_{A(t),w(t)}\qquad\text{for all }t\in [0,T],
    \end{equation}
    where $A(t)$ is the set introduced in \eqref{eq:Asigma}. In particular, \ref{def:CO} and \ref{def:ID} are fulfilled and $u(t)$ belongs to $H^1_{\Gamma,w(t)}(\Omega, A(t))^+\cap L^\infty(\Omega)\cap C^{0,1}_{\rm loc}(\Omega)$ for all $t\in [0,T]$. 
\end{prop}
\begin{proof}
If $t=0$, by \ref{def:IDeps} and recalling that $H^1_{\Gamma, w(0)}(\Omega, A_0)\subseteq H^1_{\Gamma, w(0)}(\Omega, A^\varepsilon_0)$ we deduce
\begin{equation}\label{eq:ineq0}
    \frac 12\int_\Omega |\nabla u^\varepsilon(0)|^2\, dx\le \frac 12\int_\Omega |\nabla \mathfrak{h}_{A_0,w(0)}|^2\, dx,
\end{equation}
whence by \Cref{lemma:minimiser_dirichlet} we infer the uniform bounds
\begin{equation}    \|u^\varepsilon(0)\|_{L^\infty(\Omega)}+\|u^\varepsilon(0)\|_{H^1(\Omega)}\le C.
\end{equation}
Up to a subsequence, we thus obtain that \eqref{eq:convergenceeps0} holds true for a function $u(0)\in H^1_{\Gamma,w(0)}(\Omega)^+\cap L^\infty(\Omega)$, but the convergence in $H^1(\Omega)$ is just weak at the moment. Since $u^\varepsilon(0)$ vanishes outside $A_0^\varepsilon$, one easily deduces that $u(0)\in H^1_{\Gamma,w(0)}(\Omega,A_0)$, and so by \eqref{eq:ineq0} we can characterize it as $u(0)=\mathfrak{h}_{A_0,w(0)}$. In particular, \ref{def:ID} and both \ref{def:CO} and \ref{def:GS} at time $t=0$ hold true, since $A_0$ is stable by assumption, and \eqref{eq:convergenceeps0} is true without passing to a subsequence. Moreover, exploiting \eqref{eq:ineq0}, we infer
\begin{equation}
     \limsup\limits_{\varepsilon\to 0}\frac 12\int_\Omega |\nabla u^\varepsilon(0)|^2\, dx\le \frac 12\int_\Omega |\nabla \mathfrak{h}_{A_0,w(0)}|^2\, dx,
\end{equation}
whence the convergence is actually strong in $H^1(\Omega)$.

We now focus on the positive times $t\in(0,T]$. Arguing as in \Cref{prop:boundDES}, by means of \ref{def:GSeps} we deduce:
    \begin{enumerate}[(i)]
    \item $\sup\limits_{t\in (0,T]}\|u^\varepsilon(t)\|_{L^\infty(\Omega)}\le M$;
    \item $\sup\limits_{t\in (0,T]}\|u^\varepsilon(t)\|_{H^1(\Omega)}\le C$;
    \item for all $\Omega'$ well contained in $\Omega$, there exists $C'>0$ such that $\sup\limits_{t\in (0,T]}\|u^\varepsilon(t)\|_{C^{0,1}(\Omega')}\le C'$.
\end{enumerate}
We fix $t\in (0,T]$ and, up to taking a further subsequence $\varepsilon_n$ possibly depending on $t$, we infer that \eqref{eq:convergenceeps} holds true for a function $u(t)\in H^1_{\Gamma,w(t)}(\Omega)^+\cap L^\infty(\Omega)\cap C^{0,1}_{\rm loc}(\Omega)$, except that the convergence in $H^1(\Omega)$ is just weak. In particular, \ref{def:CO} is satisfied. By arguing exactly as in the proof of \eqref{eq:claimcontained}, one can show that
\begin{equation}
    \{ u(t)>0\}\subseteq A(t),
\end{equation}
whence also \eqref{eq:AuinA} holds true.

In order to prove \ref{def:GS}, we fix $v\in H^1_{\Gamma,w(t)}(\Omega)$ and by passing to the limit in \ref{def:GSeps} we obtain
\begin{align}
    \frac 12 \int_\Omega |\nabla u(t)|^2\, dx&\le \liminf\limits_{n\to +\infty}\frac 12 \int_\Omega |\nabla u^{\varepsilon_n}(t)|^2\, dx\le \frac 12 \int_\Omega |\nabla v|^2\, dx+\lim\limits_{n\to {+}\infty}\int_\Omega \, \kappa \chi_{\{v>0\}}(1{-}\chi_{A_{u^\varepsilon_n(t)}}) dx\\
    &= \frac 12 \int_\Omega |\nabla v|^2\, dx+\int_\Omega \, \kappa \chi_{\{v>0\}}(1-\sigma(t)) dx.
\end{align}
Exploiting the inequalities $\sigma(t)\ge \chi_{A(t)}\ge \chi_{A_{u(t)}}$, by means of \Cref{lemma:equivalent} we finally deduce \eqref{eq:u=uAeps}, \ref{def:GS}, and the fact that the stated convergence of $u^\varepsilon(t)$ to $u(t)$ holds for the whole sequence.

We just need to prove that such convergence is also strong in $H^1(\Omega)$. To this aim, we exploit \ref{def:GSeps} which yields
\begin{align}
    \limsup\limits_{\varepsilon\to 0}\frac 12 \int_\Omega |\nabla u^{\varepsilon}(t)|^2&\le  \frac 12 \int_\Omega |\nabla u(t)|^2\, dx+\limsup\limits_{\varepsilon\to 0}\int_{\{u(t)>0\}\setminus A_{u^\varepsilon(t)}}\kappa \, dx\\
    &=\frac 12 \int_\Omega |\nabla u(t)|^2\, dx+\int_{\{u(t)>0\}}\kappa (1-\sigma(t))\, dx\\
    &\le \frac 12 \int_\Omega |\nabla u(t)|^2\, dx+\int_{\{u(t)>0\}\setminus A_{u(t)}}\kappa \, dx=\frac 12 \int_\Omega |\nabla u(t)|^2\, dx,
\end{align}
and we conclude.
\end{proof}

\subsection{Proof of \Cref{thm:conditional}}
We conclude the section by proving \Cref{thm:conditional}, which provides a DES of the debonding model. We start by showing that, under uniform global Lipschitz bounds, it is possible to pass to the limit in the energy inequality \ref{def:EIeps}, in particular in the work of the prescribed displacement.

\begin{prop} \label{p:cond1}
    Assume that $w\in AC([0,T];W^{2,p}(\Omega))$ for some $p>d$ and that it satisfies $\min_\Gamma w(t)>0$ for a.e. $t\in [0,T]$. If the following uniform bound holds true
    \begin{equation}\label{eq:Lipbounds}
        \operatorname{ess\, sup}\limits_{t\in [0,T]}\|u^\varepsilon(t)\|_{C^{0,1}(\Omega)}\le C,
    \end{equation}
    then the limit function $u$ obtained in \Cref{teo:convergence_ueps} satisfies the following properties:
    \begin{enumerate}[label = (\roman*)]
        \item $A_{u(t)}\in \mathcal O_{\Gamma,\dot w(t)}$ for a.e. $t\in [0,T]$;
        \item the map
            \begin{equation}
                \displaystyle t\mapsto \int_\Omega \nabla \mathfrak{h}_{A_{u(t)},\dot w(t)}\cdot \nabla u(t) \, dx
            \end{equation}
            belongs to $L^1(0,T)$;
        \item the inequality
            \begin{equation}\label{eq:eineqbart}
                \frac 12 \int_\Omega |\nabla u(t)|^2 dx+\int_{A_{u(t)}\setminus A_0}\kappa\,d x\le   \frac 12 \int_\Omega |\nabla u(0)|^2 dx+\int_0^t\int_\Omega \nabla \mathfrak{h}_{A_{u(\tau)},\dot w(\tau)} \cdot \nabla u(\tau) \, dx\, d\tau
            \end{equation}
        holds true for all $t\in [0,T]$.
    \end{enumerate}
\end{prop}

\begin{proof}
    Passing to the limit \ref{def:EIeps} and recalling that $\sigma(t)\ge \chi_{A_u(t)}$, for all $t\in [0,T]$ we know that
    \begin{equation}\label{eq:ei}
        \frac 12 \int_\Omega |\nabla u(t)|^2 dx+\int_{A_{u(t)}\setminus A_{0}}\kappa\,d x \le  \frac 12 \int_\Omega |\nabla u(0)|^2 dx+\liminf\limits_{\varepsilon \to 0}\int_0^t \int_\Omega \nabla \dot w^\varepsilon(\tau)\cdot \nabla u^\varepsilon(\tau) \, dx\, d\tau.
    \end{equation}

    In order to deal with the above liminf, we first observe that \eqref{eq:Lipbounds} implies that for a.e. $t\in [0,T]$ the functions $u^\varepsilon (t)$ uniformly converge in the whole of $\Omega$ to $u(t)$, which in particular belongs to $C^{0,1}(\Omega)$. Then, for a.e. $t\in [0,T]$, by using \eqref{eq:nablaphi} we estimate
    \begin{equation}\label{eq:L1bound}
    \begin{aligned}
        \left|\int_\Omega \nabla \dot w^\varepsilon(t)\cdot \nabla u^\varepsilon(t) \, dx\right|&\le \int_{\Gamma^\varepsilon} |\nabla u^\varepsilon (t)|\left(|\Phi^\varepsilon||\nabla \dot w(t)|{+}|\nabla \Phi^\varepsilon||\dot w(t)|\right)\, dx\\
        &\le C\int_{\Gamma^\varepsilon}\left(|\nabla \dot w(t)|+\frac 1\varepsilon|\dot w(t)|\right) \, dx
        \\&\le C\left(\varepsilon \|\nabla \dot w(t)\|_{L^\infty(\Omega)}{+}\| \dot w(t)\|_{L^\infty(\Omega)}\right)\\
        &\le C  \| \dot w(t)\|_{W^{1,\infty}(\Omega)}\le C  \| \dot w(t)\|_{W^{2,p}(\Omega)},
    \end{aligned}         
    \end{equation}
    where in the last inequality we exploited the Sobolev Embedding Theorem. Since
    \begin{equation}
        t\mapsto\| \dot w(t)\|_{W^{2,p}(\Omega)}\in L^1(0,T),
    \end{equation}
    by means of the Reverse Fatou's Lemma we infer
    \begin{equation}\label{eq:revFatou}
        \limsup\limits_{\varepsilon \to 0}\int_0^t \int_\Omega \nabla \dot w^\varepsilon(\tau)\cdot \nabla u^\varepsilon(\tau) \, dx\, d\tau\le \int_0^t  \limsup\limits_{\varepsilon \to 0}\int_\Omega \nabla \dot w^\varepsilon(\tau)\cdot \nabla u^\varepsilon(\tau) \, dx \,d\tau\leq C.
    \end{equation}
    We now claim that for a.e. $t\in [0,T]$ there exists $\varepsilon_t>0$ such that
    \begin{equation}\label{eq:containepsilon}
        \{u(t)>\delta_t\}\subseteq A_{u^{\varepsilon}(t)}\quad\text{for all $\varepsilon\le \varepsilon_t$},
    \end{equation}
    where $\delta_t:= \frac12 \min\limits_\Gamma w(t)$. Indeed, by choosing $\varepsilon_t$ in such a way that $\|u^\varepsilon(t)-u(t)\|_{C^0(\Omega)}\le \delta_t/2$ for $\varepsilon\le \varepsilon_t$, one easily deduces that
    \begin{equation}
        u^\varepsilon(t,x)\ge u(t,x)-\frac{\delta_t}{2}\quad\text{for all $x\in \Omega$ and $\varepsilon\le \varepsilon_t$,}
    \end{equation}
    whence
    \begin{equation}
        \{u(t)>\delta_t\}\subseteq\left\{u(t)>\frac{\delta_t}{2}\right\}\subseteq\{u^\varepsilon(t)>0\}\subseteq A_{u^{\varepsilon}(t)}\quad\text{for all $\varepsilon\le \varepsilon_t$}.
    \end{equation}
    Since $u(t)=w(t)\ge 2\delta_t$ on $\Gamma$ and since $u(t)$ is continuous up to the boundary of $\Omega$, we infer that $\{u(t)>\delta_t\}$ contains an open neighborhood of $\Gamma$, and so in particular it belongs to $\mathcal O_{\Gamma,\dot w(t)}$. From \eqref{eq:containepsilon} we hence obtain that for a.e. $t\in (0,T)$ the function $\mathfrak{h}_{\{u(t)>\delta_t\},\dot w(t)}$ belongs to $H^1_{\Gamma,\dot w(t)}(\Omega, A_{u^{\varepsilon}(t)})$ for $\varepsilon\le \varepsilon_t$. Recalling that $u^\varepsilon(t)=\mathfrak{h}_{A_{u^\varepsilon(t)},w(t)}$, by exploiting \eqref{eq:euler_lagrange_dirichlet} we finally deduce
    \begin{equation}
        \lim\limits_{\varepsilon \to 0}\int_\Omega \nabla \dot w^\varepsilon(t)\cdot \nabla u^\varepsilon(t) \, dx =\lim\limits_{\varepsilon \to 0}\int_\Omega \nabla \mathfrak{h}_{\{u(t)>\delta_t\},\dot w(t)}\cdot \nabla u^\varepsilon(t) \, dx =\int_\Omega \nabla \mathfrak{h}_{\{u(t)>\delta_t\},\dot w(t)} \cdot \nabla u(t) \, dx. 
    \end{equation}
    Observing that $\{u(t)>\delta_t\}\subseteq \{u(t)>0\}\subseteq A_{u(t)}$, for a.e. $t\in [0,T]$ we have $A_{u(t)}\in \mathcal O_{\Gamma,\dot w(t)}$, and so we can continue the previous chain of identities obtaining
    \begin{equation}\label{eq:lastpassage}
        \lim\limits_{\varepsilon \to 0}\int_\Omega \nabla \dot w^\varepsilon(t)\cdot \nabla u^\varepsilon(t) \, dx =\int_\Omega \nabla \mathfrak{h}_{A_{u(t)},\dot w(t)} \cdot \nabla u(t) \, dx.
    \end{equation}
    Also notice that the last integral above is in $L^1(0,T)$ since it is the pointwise countable limit of terms bounded in $L^1(0,T)$, see \eqref{eq:L1bound}. Putting together \eqref{eq:ei}, \eqref{eq:revFatou}, and \eqref{eq:lastpassage} we finally conclude.
\end{proof}

We now conclude the proof of \Cref{thm:conditional} showing \ref{def:EB}.

\begin{prop}
    Under the assumptions of \Cref{p:cond1}, the limit function $u$ obtained in \Cref{teo:convergence_ueps} satisfies \ref{def:EB}. 
\end{prop}
\begin{proof}
    We just need to show that $u$ fulfils the opposite inequality of \eqref{eq:eineqbart}. First, recall that by \eqref{eq:Lipbounds} we know that $u(t)$ is continuous up to the boundary for a.e. $t\in\Omega$, hence by the positivity of $w$ we get that $A_{u(s)}$ contains a neighbourhood of $\Gamma$ for all $s\in(0,T]$. In particular, it is possible to construct a function $ w_{s}$ satisfying \eqref{eq:safeassumption} in $[s, T]$ with $A_{u(s)}$ in place of $A_0$.   
    Setting $$P(t):=\int_\Omega \nabla \mathfrak{h}_{A_{u(t)},\dot w(t)} \cdot \nabla u(t) \, dx,$$ by means of \Cref{prop:upperineq}, we know that 
    \begin{equation}
        \frac 12 \int_\Omega |\nabla u(t)|^2 dx+\int_{A_{u(t)}\setminus A_{u(s)}}\kappa\,d x\ge   \frac 12 \int_\Omega |\nabla u(s)|^2 dx+\int_s^t P(\tau)\, d\tau\qquad\text{for all }0<s\leq  t\le T.
    \end{equation}
    For all $t\in (0,T]$, recalling that $\kappa=0$ on $A_0$ we thus deduce 
    \begin{equation}\label{eq:limsup}
        \begin{aligned}
        \frac 12 \int_\Omega |\nabla u(t)|^2 dx+\int_{A_{u(t)}\setminus A_{0}}\kappa\,d x-\int_0^t P(\tau)\, d\tau
        \\\ge \limsup\limits_{s\to 0}\left(\frac 12 \int_\Omega |\nabla u(s)|^2 dx+\int_{A_{u(s)}}\kappa\,d x\right).
    \end{aligned}
    \end{equation}
    By exploiting \ref{def:GS}, we know that $u(s)$ is uniformly bounded in $H^1(\Omega)$, so there exists a subsequence $u(s_n)$ weakly converging to a function $\overline u\in H^1_{\Gamma,w(0)}(\Omega)^+$ as $n\to +\infty$. In particular, since $u(s_n)$ vanishes outside $A_{u(s_n)}$, we observe that
    \begin{equation}
        \overline u=0\quad\text{outside}\quad \overline A:=\bigcap_{n\in \n}A_{u(s_n)}.
    \end{equation}
    Since $u(0)$ is stable we now infer
    \begin{align}
        \frac 12 \int_\Omega |\nabla u(0)|^2 dx&\le \frac 12 \int_\Omega |\nabla \overline u|^2 dx+\int_{\{\overline u>0\}}\kappa\,d x\le \frac 12 \int_\Omega |\nabla \overline u|^2 dx+\int_{\overline A
        }\kappa\,d x\\
        &\le\liminf\limits_{n\to +\infty}\left(\frac 12 \int_\Omega |\nabla u(s_n)|^2 dx+\int_{A_{u(s_n)}}\kappa\,d x\right)\\
        &\le \limsup\limits_{s\to 0}\left(\frac 12 \int_\Omega |\nabla u(s)|^2 dx+\int_{A_{u(s)}}\kappa\,d x\right).
    \end{align}
    Combining the above inequality with \eqref{eq:limsup} we conclude.
\end{proof}

	\bigskip
	
	\noindent\textbf{Acknowledgements.} The authors are grateful to Gianni Dal Maso, Dario Mazzoleni and Bozhidar Velichkov for many fruitful discussions on the topic. The authors are members of GNAMPA (INdAM). F. R. and E. G. T. have been partially supported by the INdAM-GNAMPA project 2025 \lq\lq DISCOVERIES\rq\rq (CUP E5324001950001). E. M. and E. G. T. acknowledge support from PRIN 2022 (Project no. 2022J4FYNJ), funded by MUR, Italy, and the European Union -- Next Generation EU, Mission 4 Component 1 CUP F53D23002760006.

	\bigskip
    \printbibliography

	\medskip
	\small
	\begin{flushleft}
		\noindent \verb"eleonora.maggiorelli01@universitadipavia.it"\\
		Dipartimento di Matematica  ``F. Casorati'', \\
		Universit\`a di Pavia,\\
		via Ferrata 5, 27100 Pavia, Italy\\
		\smallskip
		\noindent \verb"filippo.riva@unibocconi.it"\\
		Department of Decision Sciences, \\
		Bocconi University,\\
		via Roentgen 1, 20136 Milano, Italy\\
        \smallskip
        \noindent \verb"edoardogiovann.tolotti01@universitadipavia.it"\\
		Dipartimento di Matematica  ``F. Casorati'', \\
		Universit\`a di Pavia,\\
		via Ferrata 5, 27100 Pavia, Italy\\
	\end{flushleft}

\end{document}